\documentclass[11pt, a4paper, reqno]{amsart}
\usepackage[T1]{fontenc}
\usepackage[latin1]{inputenc}
\usepackage[english]{babel}
\usepackage[reqno]{amsmath}
\usepackage{amsthm}
\usepackage{latexsym}
\usepackage{amssymb}
\usepackage{mathrsfs}
\usepackage{mathtools}
\usepackage{mathabx}
\usepackage{lipsum}
\usepackage{titlesec}
\usepackage{tikz-cd}
\usepackage{dsfont} 

\usepackage[totalwidth=15cm,totalheight=20cm, hmarginratio=1:1]{geometry}
\usepackage{mathtools}
\usepackage{color}
\usepackage{physics}
\usepackage{faktor}
\usepackage[shortlabels]{enumitem}
\usepackage{tikz} 
\usepackage{hyperref}
\usepackage{graphicx}
\usepackage{tabularx}
\usepackage{cases} 
\newcolumntype{C}{>{\centering\arraybackslash}X} 
\usepackage{csquotes}
\usepackage{url}

\usepackage{accents}

\makeatletter
\newtheorem*{rep@theorem}{\rep@title}
\newcommand{\newreptheorem}[2]{%
\newenvironment{rep#1}[1]{%
 \def\rep@title{#2 \ref{##1}}%
 \begin{rep@theorem}}%
 {\end{rep@theorem}}}
\makeatother

\makeatletter
\newtheorem*{rep@cor}{\rep@title}
\newcommand{\newrepcor}[2]{%
\newenvironment{rep#1}[1]{%
 \def\rep@title{#2 \ref{##1}}%
 \begin{rep@cor}}%
 {\end{rep@cor}}}
\makeatother

\makeatletter
\newtheorem*{rep@prop}{\rep@title}
\newcommand{\newrepprop}[2]{%
\newenvironment{rep#1}[1]{%
 \def\rep@title{#2 \ref{##1}}%
 \begin{rep@prop}}%
 {\end{rep@prop}}}
\makeatother

\newtheorem{theorem}{Theorem}[section]
\newreptheorem{theorem}{Theorem}
\numberwithin{theorem}{section}

\newtheorem{manualtheoreminner}{Theorem} 
\newenvironment{manualtheorem}[1]{%
    \renewcommand\themanualtheoreminner{#1}%
    \manualtheoreminner
}{\endmanualtheoreminner}

\newtheorem{lemma}[theorem]{Lemma}
\newtheorem{corollary}[theorem]{Corollary}

\newrepcor{corollary}{Corollary}

\newtheorem{proposition}[theorem]{Proposition}

\newrepprop{prop}{Proposition}

\theoremstyle{definition}
\newtheorem*{definition*}{Definition}
\newtheorem{definition}[theorem]{Definition}

\theoremstyle{remark}
\newtheorem{remark}[theorem]{Remark}

\makeatletter
\def\paragraph{\@startsection{paragraph}{4}%
  \z@\z@{-\fontdimen2\font}%
  {\normalfont\bfseries}}
\makeatother

\numberwithin{equation}{section}
\allowdisplaybreaks

\patchcmd{\subsection}{-.5em}{.5em}{}{}

\makeatletter

\renewcommand\section{\@startsection{section}{1}%
  \z@{.7\linespacing\@plus\linespacing}{.5\linespacing}%
  {\normalfont\scshape\centering}}

\renewcommand\subsection{\@startsection{subsection}{2}%
  \z@{-.5\linespacing\@plus-.7\linespacing}{.5\linespacing}%
  {\bfseries}}
  
\renewcommand\subsubsection{\@startsection{subsubsection}{3}%
  \z@{-.5\linespacing\@plus-.7\linespacing}{.5\linespacing}%
  {\itshape}}
  
\def\l@paragraph{\@tocline{4}{0pt}{1pc}{7pc}{}}

\makeatother

\newcommand{\C}{\mathbb{C}}
\newcommand{\R}{\mathbb{R}}

\newcommand{\Z}{\mathbb{Z}}

\newcommand{\dsone}{\mathds{1}}

\newcommand{\Hyp}{\mathbb{H}^2}

\newcommand{\Teich}{\mathcal{T}}

\newcommand{\Lsl}{\mathfrak{sl}}
\newcommand{\Dlie}{\mathcal{L}}

\newcommand{\Vol}{\mathrm{Vol}}

\newcommand{\vl}{|}

\newcommand{\PSL}{\mathbb{P}\mathrm{SL}}

\newcommand{\Id}{\mathrm{Id}}

\newcommand{\Ree}{\mathcal{R}e}

\renewcommand{\i}{\mathbf{I}}

\newcommand{\g}{\mathbf{g}}

\newcommand{\deft}{\mathcal{B}_0(T^2)}

\newcommand{\defg}{\mathcal{B}(\Sigma_g)}
\newcommand{\defgp}{\mathcal{B}_0(\Sigma_g)}
\newcommand{\RP}{\mathbb{R}\mathbb{P}^2}
\newcommand{\Sg}{\Sigma_g}
\newcommand{\dev}{\mathrm{dev}}
\newcommand{\ome}{\boldsymbol{\omega}}
\newcommand{\sol}{$\frac{1}{3}\log(2\vl c\vl^2)$}
\newcommand{\cubic}{Q^3(\Teich^c(T^2))}
\newcommand{\almost}{\mathcal{J}(\R^2)}
\newcommand{\pick}{D^3(\mathcal{J}(\R^2))}
\newcommand{\normpick}{\vl\vl A\vl\vl_0^2}
\newcommand{\cubichyp}{Q^3(\mathbb{H}^2)}
\newcommand{\dx}{\mathrm{d}x}
\newcommand{\dy}{\mathrm{d}y}
\newcommand{\du}{\mathrm{d}u}
\newcommand{\devu}{\mathrm{d}v}

\DeclarePairedDelimiterX{\scal}[2]{\langle}{\rangle}{#1 \mid #2}
\DeclarePairedDelimiterX{\scall}[2]{\langle}{\rangle}{#1, #2}

\DeclareMathOperator{\Imm}{Im}
\DeclareMathOperator{\Span}{Span}
\DeclareMathOperator{\Hom}{Hom}
\DeclareMathOperator{\End}{End}
\DeclareMathOperator{\Aut}{Aut}

\DeclareMathOperator{\SL}{\mathrm{SL}}

\DeclareMathOperator{\Proj}{\mathrm{Proj}}

\DeclareMathOperator{\Iso}{Iso}

\DeclareMathOperator{\Ad}{Ad}

\begin{document}

\setcounter{secnumdepth}{3}
\setcounter{tocdepth}{2}

\title[Pseudo-K\"ahler Geometry of properly convex \texorpdfstring{$\RP$}{RP2}-structures on \texorpdfstring{$T^2$}{T2}]{Pseudo-K\"ahler Geometry of Properly Convex Projective Structures on the torus}

\author[Nicholas Rungi]{Nicholas Rungi}
\address{NR: Scuola Internazionale Superiore di Studi Avanzati (SISSA), Trieste (TS), Italy.} \email{nrungi@sissa.it} 

\author[Andrea Tamburelli]{Andrea Tamburelli}
\address{AT: Department of Mathematics, Rice University, Houston (TX), USA.} \email{andrea.tamburelli@libero.it}

\date{\today}

\begin{abstract} In this paper we prove the existence of a pseudo-K\"ahler structure on the deformation space $\deft$ of properly convex $\RP$-structures over the torus. In particular, the pseudo-Riemannian metric and the symplectic form are compatible with the complex structure inherited from the identification of $\deft$ with the complement of the zero section of the total space of the bundle of cubic holomorphic differentials over the Teichm\"uller space. We show that the $S^1$-action on $\deft$, given by rotation of the fibers, is Hamiltonian and it preserves both the metric and the symplectic form. Finally, we prove the existence of a moment map for the $\SL(2,\R)$-action over $\deft$.
\end{abstract}

\maketitle

\tableofcontents

\section{Introduction}
The purpose of this paper is to study the geometry of the deformation space of properly convex $\RP$-structures over the torus. This work is inspired by \cite{mazzoli2021parahyperkahler}, where the authors developed the techniques to study the geometry of the deformation space of MGHC anti-de Sitter structures over $\Sg\times\R$, and a similar approach is used in our case.

\subsection{Motivation}
The study of $\R\mathbb{P}^n$-structures on a smooth manifold started in the 30s with the work of Ehresmann. The case $n=2$ arose great interest among the mathematical community, to the point that between the 50s and the 60s the first results were obtained for the classification of (convex) $\RP$-structures on closed surfaces. See for example \cite{kuiper1953convex} and \cite{benzecri1960varietes}. About 30 years later, Goldmann \cite{goldman1990convex} showed that the deformation space of convex $\RP$-structures $\defg$ over a closed surface $\Sg$, with $g\ge 2$, is a cell of real dimension $16g-16$. He proved a more general result for compact surfaces with $k$ boundary components and he also described the geometry of the deformation space of not necessarily convex $\RP$-structures. On the other hand, the space $\defg$ is intimately related to the Teichm\"uller space $\Teich^c(\Sg)$. In fact, using the Klein-Beltrami model for hyperbolic geometry, every hyperbolic structure on $\Sg$ induces a convex $\RP$-structure, giving an embedding of $\Teich^c(\Sg)$ inside $\defg$.\newline Loftin proved in \cite{loftin2001affine}, using the theory of hyperbolic affine spheres, that the space $\defg$ can be endowed with a complex structure coming from the identification with the total space of the holomorphic bundle of cubic differentials over $\Teich^c(\Sg)$. It has been conjectured that this complex structure is compatible with a natural symplectic form defined on $\defg$, which restricts to a multiple of the Weil-Petersson K\"ahler form on $\Teich^c(\Sg)$, see \cite{goldman1990symplectic}. It is not yet known whether these two structures matched together give a Riemannian metric. Recently, it has been shown (\cite{kim2017kahler}) that $\defg$ admits an invariant mapping class group K\"ahler metric, but it is not clear how this is related to the natural complex structure introduced by Loftin.\newline In the case of genus one, the holomorphic vector bundle of $k$-differentials $Q^k(\Teich^c(T^2))$ has complex rank equal to one. Thus, its total space can be identified with $\Hyp\times\C$, where here $\Hyp$ has to be interpreted as a copy of $\Teich^c(T^2)$. Recently, it has been shown (\cite{trautwein2018infinite}) that the total space of the disc bundle inside $Q^k(\Teich^c(T^2))$ admits a family of symplectic forms parametrized by smooth functions $f:[0,1)\to[0,1)$, which have nice geometric descriptions if the functions are chosen carefully. This result extends the Donaldson's construction of the Feix-Kaledin hyperK\"ahler metric (the case $k=2$ and $f(t)=1-\sqrt{1-t}$) defined on a neighborhood of the zero section of $T^*\Hyp$ (\cite{donaldson2003moment}), where $T^*\Hyp$ is identified with the holomorphic bundle of quadratic differentials over $\Teich^c(T^2)$. 
\newline The aim of this paper is to show that, in the case $k=3$, as long as the condition of being positive-definite on the metric is relaxed, it is possible to find a family of symplectic forms $\ome_f$ and pseudo-Riemannian metrics $\g_f$, depending on a choice of a function $f:[0,\infty)\to(-\infty,0]$ with some properties, on the whole $Q^3(\Teich^c(T^2))$. In particular, they are compatible with the natural complex structure on $Q^3(\Teich^c(T^2))$, thus giving a family of pseudo-K\"ahler structures to $Q^3(\Teich^c(T^2))$. We will see how the pseudo-K\"ahler structure on the space of properly convex $\RP$-structures over the torus will be induced from the one on $\cubic$, 
and we will give a geometric description of this process in terms of hyperbolic affine spheres.\newline We now introduce the fundamental notions and we state our main theorems.

\subsection{Deformation space of convex \texorpdfstring{$\RP$}{RP2}-structures}
Here we give the standard definitions of \emph{convex} $\RP$-\emph{structures} over a smooth surface $S$ and we introduce the deformation space of such structures. Most of the notions presented here can be found in \cite{goldman1990convex}, \cite{choi1997classification} and \cite{goldman2018geometric}.\newline An $\RP$-\emph{structure} on a smooth connected surface $S$ is a maximal $\RP$-atlas, namely an atlas in which the local charts take value in the real projective plane and the transition functions restrict to projective transformations on each connected component of the subset where defined. Once a maximal $\RP$-atlas is given, we say that $S$ has an $\RP$-\emph{structure}. By unravelling the definition it is easy to see that if $S$ is an $\RP$-surface and $p:\widetilde{S}\to S$ is its universal cover, then $\widetilde{S}$ inherits an $\RP$-structure from the one of $S$.\newline A domain (open and connected) $\Omega\subset\RP$ is said to be \emph{convex} if there exists a projective line $l$ disjoint from $\Omega$ such that $\Omega\subset\RP\setminus l\cong\mathbb{A}^2$ is convex in the usual sense. By definition $\R^2$ is convex but $\RP$ is not. It is not difficult to show that this notion of convexity in the real projective plane does not add new convex sets with respect to those usual in affine spaces, see \cite[\S 1]{andersson2004complex} .
\begin{definition}
An $\RP$-surface $S$ is \emph{convex} if it is projectively isomorphic to a quotient $\Omega/_{\Gamma}$, where $\Omega\subset\RP$ is a convex domain and $\Gamma\subset\Proj(\Omega)\subset\SL(3,\R)$ is a discrete group of projective transformations preserving $\Omega$ acting freely and properly discontinuously on $\Omega$. The surface $S$ is $\emph{properly convex}$ if $\Omega$ is bounded in some affine space.
\end{definition}

There is a well-known equivalent way of defining a convex $\RP$-surface in terms of the existence of a pair of maps with special properties. This is the following:
\begin{theorem}[Development Theorem]
Let $S$ be an $\RP$-surface, then the following are equivalent:
\begin{itemize}
    \item[(1)] $S$ is convex
    \item[(2)] There exists a pair $(\dev, h)$, where $\dev:\widetilde{S}\to\RP$ is a diffeomorphism onto a convex domain in $\RP$ called the developing map and $h:\pi_1(S)\to\SL(3,\R)$ is a group homomorphism called the holonomy homomorphism, such that the following diagram commutes:\begin{equation}
\begin{tikzcd}
\widetilde{S} \arrow[r, "\mathrm{dev}"] \arrow[d, "\gamma"'] & \mathbb{R}\mathbb{P}^2 \arrow[d, "h(\gamma)"] \\
\widetilde{S} \arrow[r, "\mathrm{dev}"]                         & \mathbb{R}\mathbb{P}^2                       
\end{tikzcd}
\end{equation}
Moreover, if $(\widetilde{dev},\widetilde{h})$ is another such pair, then $\exists g\in\SL(3,\R)$ such that:\begin{equation*}
    \widetilde{\dev}=g\circ\dev, \qquad \widetilde{h}(\gamma)=g\circ h(\gamma)\circ g^{-1}, \ \forall\gamma\in\pi_1(S) \ . 
\end{equation*}
\end{itemize}
\end{theorem}
It is clear from the statement that if $S$ is convex, then its universal cover $\widetilde{S}$ can be identified with a convex domain $\Omega\subset\RP$ via the developing map and the discrete subgroup $\Gamma$ can be identified with $\pi_1(S)$ via the holonomy homomorphism. From this point on we will focus only on the case in which the surface is closed and orientable, hence it will be denoted with $\Sg$. \begin{definition}\label{defproperlyconvexstructure}Let $\Sg$ be a smooth, closed and orientable surface. A \emph{(properly) convex} $\RP$-\emph{structure} on $\Sg$ is a pair $(f,M)$, where $f:\Sg\to M$ is a diffeomorphism (called the \emph{marking}) and $M\cong \Omega/_{\Gamma}$ is a (properly) convex $\RP$-surface. 
\end{definition}One can define an equivalence relation on such pairs, namely $(f_1,M_1)\sim (f_2,M_2)$ if and only if there exists a projective isomorphism $\Psi:M_1\to M_2$ such that the new marking $\Psi\circ f_1$ is isotopic to $f_2$. Now we are ready to introduce the main space that we are going to study in this article: \emph{the deformation space of (properly) convex $\RP$-structures}
\begin{align}
    &\defg:=\{(f,M) \ \text{convex} \ \RP-\text{structure on} \ \Sg\}_{/_\sim} \\ &\defgp:=\{(f,M) \ \text{properly convex} \ \RP-\text{structure on} \ \Sg\}_{/_\sim}
\end{align}
\begin{remark}
The reader familiar with Teichm\"uller theory may have noticed the remarkable similarity of the construction of this space with $\Teich^c(\Sg)$, when $g\ge 2$. In fact, as stated before, if $\Sg\cong\Hyp/_{\Gamma}$ with $\Gamma\subset\PSL(2,\R)$, using the Klein-Beltrami model for hyperbolic geometry and noting that $\PSL(2,\R)\cong\Iso^+(\Hyp)$ and $\PSL(3,\R)\cong\Aut(\RP)$, we can identify $\Gamma$ with a discrete subgroup of $\PSL(3,\R)$, contained in $\Proj(\Omega)$, acting freely and properly discontinuously on $\Omega\equiv\Hyp$. Thus, giving an inclusion $\Teich^c(\Sg)\subset\defg$ at least as sets.\footnote{We have implicitly used the existence of a unique (up to conjugation) irreducible representation $\imath:\PSL(2,\R)\hookrightarrow\PSL(3,\R)$.}
\end{remark}
The behavior of this space depends highly on the genus of the surface and, as one can imagine, there are notable differences between the flat case (genus one) and the hyperbolic one ($g\ge 2$). The first result in this direction was given by Kuiper and Benzecr\'i (\cite{kuiper1953convex},\cite{benzecri1960varietes}) which we now recall.
\begin{proposition}
If $\Sg$ is a convex $\RP$-surface with $g\ge 2$, then it must be properly convex. Moreover, the boundary $\partial\Omega$ is always strictly convex and $C^1$, and it must be either and ellipse or a Jordan curve which is nowhere $C^2$. In particular, there is an identification $\defg\equiv\defgp$.
\end{proposition}
In the case of the torus this is no longer true, for instance there are many convex $\RP$-structures which are not properly convex: affine and Euclidean ones. They can not be properly convex since the developing map identifies the universal cover of $T^2$ with a copy of $\mathbb{R}^2$ inside $\RP$, which is convex but not bounded, see \cite[\S 8.5]{goldman2018geometric}.  \newline The most important results about this space, when $g\ge 2$, are summarized in these two theorems:
\begin{theorem}[Goldman \cite{goldman1990convex} and Goldman-Choi \cite{choi1993convex}]
The deformation space $\defg$ is a smooth manifold diffeomorphic to $\R^{16g-16}$. Moreover, it is isomorphic to one of the three connected components of the $\SL(3,\R)$-character variety, called the Hitchin component.
\end{theorem}

\begin{theorem}[Loftin \cite{loftin2001affine}]
The data of a compact Riemann surface $(\Sg, J)$, with $g\ge 2$, together with a holomorphic section of the tri-canonical bundle $K^3$ is equivalent to the data of a smooth, oriented closed surface $S$ homeomorphic to $\Sg$ with a convex $\RP$-structure.
\end{theorem}
In other words, there exists an homeomorphism \begin{align*}
    \Psi&:\defg\to Q^3(\Teich^c(\Sg)) \\ &[(f,\Omega/_{\Gamma})]\mapsto (J,q)
\end{align*}where $q\in H^0(\Sg,K^3)$. From this, $\defg$ inherits the structure of a holomorphic vector bundle over $\Teich^c(\Sg)$, in particular it inherits a natural complex structure which is invariant under the action of the mapping class group $MCG(\Sg)$. It must be noted that Loftin's theorem recovers Goldman's result about the contractility of the space and that $\Teich^c(\Sg)$ now embeds into $\defg$ as the zero section.

\subsection{Pseudo-K\"ahler geometry}
We now introduce the notion of \emph{pseudo-K\"ahler structure} and we state our main theorem.\newline
Recall that a \emph{pseudo-Riemannian} metric $\g$ on a smooth $n$-manifold $M$ is an everywhere non-degenerate, smooth, symmetric $(0,2)$-tensor. The $\emph{index}$ of $\g$ is the maximal rank $k$ of the smooth distribution where it is negative-definite. For instance, if $k=0$ then $\g$ is a Riemannian metric. Now let $\mathbf{I}$ be a complex structure on $M$, then $(\g,\mathbf{I})$ is a \emph{pseudo-Hermitian structure} if \begin{equation*}
    \g(\mathbf{I}X,\mathbf{I}Y)=\g(X,Y), \qquad\forall X,Y\in T_pM, p\in M
\end{equation*}Notice that, due to this last condition, the index of $\g$ in this case is always even $k=2s$, where $s$ is called the \emph{complex index} and it satisfies $1\le s\le m=\dim_\C M$. The \emph{fundamental 2-form} $\ome$ of a pseudo-Hermitian manifold $(M,\g,\mathbf{I})$ is defined by: \begin{equation*}
    \ome(X,Y):=\g(X,\mathbf{I}Y), \qquad\forall X,Y\in T_pM, p\in M
\end{equation*}
\begin{definition}
A pseudo-Hermitian manifold $(M,\g,\mathbf{I},\ome)$ is called \emph{pseudo-K\"ahler} if the fundamental $2$-form is closed, namely if $\mathrm{d}\ome=0$. In this case the corresponding metric is called \emph{pseudo-K\"ahler}.
\end{definition}
\begin{manualtheorem}A \label{thmA}
Let $f:[0,+\infty)\to(-\infty,0]$ be a smooth function such that:\begin{itemize}
    \item[(i)]$f'(t)<0, \ \forall t\ge0$
    \item[(ii)]$\displaystyle\lim_{t\to+\infty}f(t)=-\infty$ .
\end{itemize}Then the space $\deft$ admits a $MCG(T^2)$-invariant pseudo-K\"ahler structure $(\g_f,\ome_f,\mathbf{I})$. The pseudo-Riemannian metric $\g_f$ and the symplectic form $\ome_f$ both depend on $f$ and $\mathbf{I}$ is the complex structure coming from the identification of $\deft$ with the complement of the zero section in $Q^3(\Teich^c(T^2))$.\end{manualtheorem}
In what follows we will denote the pseudo-Riemannian metric with $\g$ and the symplectic form with $\ome$, since their dependence on $f$ will be clear from the definition (see Definition \ref{def:pseudoriemannian}).
\subsection{The circle action and the \texorpdfstring{$\SL(2,\R)$}{SL(2,R)} moment map}
We now move on to the study of two natural actions on the space $\deft$. Let us consider the complement of the zero section in the bundle of cubic holomorphic differentials: \begin{equation*}
    Q^3_0(\Teich^c(T^2)):=Q^3(\Teich^c(T^2))\setminus\big(\Teich^c(T^2)\times\{0\}\big) \ .
\end{equation*}There is an isomorphism $\chi:\deft\to Q^3_0(\Teich^c(T^2))$ (see Corollary \ref{cor:cubictorusandproperlyconvex}) which allows us to pull the natural circle action on $Q^3(\Teich^c(T^2))$, given by $e^{i\theta}\cdot(J,q)\mapsto (J,e^{-i\theta}q), \ \theta\in\R$, back to $\deft$ (notice that this circle action preserves the zero section). Let $\Psi_\theta:\deft\to\deft$ be the induced $S^1$-action, then we have the following: \begin{manualtheorem}B
The $S^1$-action on $\deft$ is Hamiltonian with respect to $\ome$ and it satisfies \begin{equation}
     \Psi_\theta^{*}\ome=\ome, \qquad \Psi_\theta^{*}\g=\g
\end{equation}\end{manualtheorem}
There is also a natural $\SL(2,\R)$-action on $Q^3(\Teich^c(T^2))$ preserving the zero section, namely if $P\in \SL(2,\R)$ then \begin{align*}
    \Phi_P: & Q^3(\Teich^c(T^2))\to Q^3(\Teich^c(T^2)) \\ &(J,q)\mapsto(PJP^{-1},(P^{-1})^*q)
\end{align*}where the action on $q$ is by pull-back. By abuse of notation we still denote with $\Phi_P$ the $\SL(2,\R)$-action on $\deft$, then we have the following:\begin{manualtheorem}C
The $\SL(2,\R)$-action on $\deft$ is Hamiltonian with respect to $\ome$ and with moment map $\mu:\deft\to\Lsl(2,\R)^*$ given by
\begin{equation}
    \langle \mu(J,q),X \rangle=\bigg(1-f\big(\vl\vl q\vl\vl^2_J\big)\bigg)\tr(JX),\quad X\in\Lsl(2,\R)
\end{equation}where if $q=c \mathrm{d}z^3$ for $c\in\C$, then $\vl\vl q\vl\vl_J^2=\vl c\vl^2e^{-3\phi}$ for a fixed metric $h=e^\phi\vl \mathrm{d}z\vl^2$ compatible with $J$.
\end{manualtheorem}

\section*{Acknowledgments}
This paper is the first part of an ongoing PhD research project of the first author under the supervision of the second author. Currently, the theory is being developed to address the problem in the case of higher genus, and a similar result is expected. 
\section{Preliminaries}
\subsection{Hyperbolic affine spheres and cubic tensors}
Here we briefly introduce the theory of hyperbolic affine spheres in general dimension, recalling also the most important result regarding their global geometry as hypersurfaces of $\R^{n+1}$. Finally, we will focus on the case $n=2$, showing that there is a close relationship between these geometric objects and the cubic holomorphic differentials defined on the surface. We will follow closely the introduction presented in \cite{loftin2013cubic}. For a more specific discussion, see \cite{nomizu1994affine}.

Affine differential geometry is the study of the differential-geometric properties of hypersurfaces in $\R^{n+1}$ which are invariant under unimodular affine transformations, namely those affine maps $T:\R^{n+1}\to\R^{n+1}$ such that $T(x)=Ax+b$, with $\det A=1$ and $b\in\R^{n+1}$. The affine invariants are not obvious at first glance, as the usual notions of length and angle in $\R^{n+1}$ are not valid.  One way of accessing the theory is the so called \emph{affine normal}, which is a transverse vector field to a smooth hypersurface in $\R^{n+1}$. In order to define the affine normal, we first discuss the affine geometry a transverse vector field generates on a hypersurface in $\R^{n+1}$.  Let $f\!: M \to \R^{n+1}$ be an immersion and $\widetilde\xi\!:M\to\R^{n+1}$ a transverse vector field to $f(M)$. This means that for all $x\in M$ we have a splitting: $$T_{f(x)}\R^{n+1} = f_* T_xM + \R\widetilde{\xi}_x \ . $$ 
Now, for any smooth vector fields $X,Y$ on $M$, push them forward by $f_*$ and locally extend $f_*X$, $f_*Y$ and $\widetilde\xi$ to smooth vector fields on $\R^{n+1}$. Then, if $D$ is the standard (flat) connection on $\R^{n+1}$ we get: 
\begin{eqnarray*}
D_{f_*X}f_*Y &=& f_*(\widetilde\nabla_X Y) + \widetilde{h}(X,Y) \widetilde\xi , \\
D_{f_*X}\widetilde\xi &=& -f_*(\widetilde{S}(X)) + \widetilde\tau(X)\widetilde\xi.
\end{eqnarray*}
One can easily check that $\widetilde\nabla$ is a torsion-free connection on $f(M)$, $\widetilde{h}$ is a symmetric $(0,2)$-tensor, $\widetilde{S}$ is an endomorphism of $TM$ and $\widetilde\tau$ is a one-form on $f(M)$, and that all the elements are independent of the extensions. Note that $D$ depends only on the affine structure on $\R^{n+1}$. Below we suppress the $f_*$ from the notation.\newline Now assume $\widetilde{h}$ is defined for every transverse vector field $\widetilde\xi$, then the affine normal $\xi$ can be defined by the following requirements:
\begin{itemize}
\item[(1)] $h:=\widetilde{h}$ is positive definite.
\item[(2)] $\widetilde\tau=0$.
\item[(3)] For each frame of tangent vectors $\{X_1,\dots, X_n\}$ to $M$, $$\det_{\R^{n+1}} (X_1,\dots,X_n,\xi)^2 = \det_{1\le i,j\le n} h(X_i,X_j).$$
\end{itemize}
The first condition means that $\xi$ points inward (so that $M$ and $\xi$ lie on the same side of the tangent plane). Moreover, if $\tilde h$ is not definite but non-degenerate, we can still define the affine normal up to a choice of orientation in a similar manner. The affine structure equations are then \begin{equation}\begin{aligned}\label{structurequations}
D_XY &= \nabla_XY + h(X,Y)\xi \\
D_X\xi &= -S(X)
\end{aligned}\end{equation}
In this case, $\nabla$ is called the \emph{Blaschke connection}, $h$ is the \emph{Blaschke metric} and $S$ is the \emph{affine shape operator}.\newline Now that we have introduced the basics of affine differential geometry we can define the affine spheres. There are several ways to do this and the one we will use will be through the affine shape operator, which is the most convenient for our discussion. \begin{definition}
Let $M$ be an immersed hypersurface in $\R^{n+1}$ with structure equations given by (\ref{structurequations}), then it is called an \emph{affine sphere} if $S=\lambda\cdot\Id_{TM}$, where $\lambda$ is a real-valued smooth function defined on $M$.
\end{definition}It is not hard to see that the function $\lambda$ has to be constant and hence, by rescaling, we can assume $\lambda\in\{-1,0,1\}$. We say that an affine sphere is \emph{parabolic} if $\lambda=0$, \emph{elliptic} if $\lambda=1$ and \emph{hyperbolic} if $\lambda=-1$. Moreover, in the last two cases we can translate the affine spheres so that $\xi=-\lambda f$, where $f:M\to\R^{n+1}$ is the immersion. For parabolic affine spheres we may apply a linear map to ensure that $\xi$ is the last coordinate vector.
\begin{remark}
In general, it is not easy to write down affine spheres except in very special cases. The simplest examples are the following: the elliptic paraboloid for $\lambda=0$, the round sphere and the ellipsoids for $\lambda=1$ and one sheet of a hyperbolic paraboloid for $\lambda=-1$. Moreover, for the parabolic and elliptic case those are the only examples of such affine spheres which are also closed subsets of $\R^{n+1}$.
\end{remark}
The global geometry of hyperbolic affine spheres is much more complicated but, at the same time, more interesting. Their properties were conjectured by Calabi (\cite{calabi1972complete}) and proved by Cheng-Yau (\cite{cheng1977regularity}, \cite{cheng1986complete}) and Calabi-Nirenberg (with clarifications by Gigena (\cite{gigena1981conjecture}) and Li (\cite{li1990calabi}, \cite{li1992calabi})). The most important result is the following: \begin{theorem}[Cheng-Yau-Calabi-Nirenberg]\label{thmchengyau}
Given a constant $\lambda<0$ and a convex, bounded domain $\Omega\subset\R^n$, there is a unique properly embedded hyperbolic affine sphere $M\subset\R^{n+1}$ with affine shape operator $S=\lambda\cdot\Id_{TM}$ and center $0$ asymptotic to the boundary of the cone $\mathcal{C}(\Omega):=\{(tx,t) \ | \ x\in\Omega, t>0\}\subset\R^{n+1}$. For any immersed hyperbolic affine sphere $f: M\rightarrow\R^{n+1}$, properness of the immersion is equivalent to the completeness of the Blaschke metric, and any such $M$ is a properly embedded hypersurface asymptotic to the boundary of the cone given by the convex hull of $M$ and its center.
\end{theorem}
 \begin{remark}\label{rmk:correspondencehypaffinesphereandpropconvex}
Theorem \ref{thmchengyau} induces a correspondence between equivariant hyperbolic affine $2$-spheres and properly convex $\RP$-structures on a closed orientable surface $\Sg$, with $g\ge 1$. A similar correspondence holds also in a more general context but it is slightly more complicated, see \cite{loftin2001affine}
\end{remark}
Before specializing in the case of surfaces, we first want to recall the definition and some properties of the so-called \emph{Pick tensor} in arbitrary dimension.\newline
Let $f:\!(M, h, \nabla)\hookrightarrow\R^{n+1}$ be an immersed affine sphere, namely an immersed hypersurface satisfying equations (\ref{structurequations}) with $S=\lambda\cdot\Id_TM$. If $\widehat{\nabla}$ denotes the Levi-Civita connection of $h$, then $\nabla=\widehat\nabla +A$, where $A$ is a section of $T^*(M)\otimes\End(TM)$. In particular, for every $X\in\Gamma(TM)$ the quantity $A(X)$ is an endomorphism of $TM$.
\begin{proposition}[{\cite[Lemma 4.3, Lemma 4.4]{benoist2013cubic}}]\label{prop:pickform}
The section $A$ has the following properties:
\begin{itemize}
    \item[(1)] $A(X)Y=A(Y)X, \qquad\forall X,Y\in\Gamma(TM)$
    \item[(2)] The endomorphism $A(X)$ is trace-free and $h$-symmetric, $\forall X\in\Gamma(TM)$
    \item[(3)] $\mathrm{d}^{\widehat\nabla}A=0$, where $\mathrm{d}^{\widehat\nabla}A(X,Y)=\big(\widehat\nabla_XA\big)(Y)-\big(\widehat\nabla_YA\big)(X),\qquad X,Y\in\Gamma(TM)$
\end{itemize}
\end{proposition}

\begin{definition}\label{def:picktensor}
The \emph{Pick tensor} is the $(0,3)$-tensor defined by \begin{equation}\label{picktensorandpickform} C(X,Y,Z):=h(A(X)Y,Z), \quad\forall X,Y,Z\in\Gamma(TM)\end{equation}
\end{definition}

\begin{corollary}\label{picktensoraffinesphere}
If $f:\!(M, h, \nabla)\hookrightarrow\R^{n+1}$ is an immersed affine sphere, then the Pick tensor is totally symmetric, namely in index notation $C_{ijk}$ we have $$C_{ijk}=C_{\sigma(ijk)},\quad\forall\sigma\in\mathfrak S_3$$
In particular, the following relation holds $$ \big(\nabla_Xh\big)\big(Y,Z\big)=-2 C(X,Y,Z) \ . $$
\end{corollary}
These last two results hold in a more general context where $M$ is not necessarily an immersed affine sphere, see \cite[\S 4.2]{benoist2013cubic}. In dimension $2$ the Pick tensor is closely related to the conformal structure of the surface: \begin{theorem}[{\cite[Lemma 4.8]{benoist2013cubic}}]\label{thm:picktensor}
Let $M$ be a smooth orientable surface endowed with a metric $h$ and a compatible complex structure $J$. Suppose the tensor $C(X,Y,Z)=h(A(X)Y,Z)$ is totally symmetric, then $A(X)$ is trace-free and $\mathrm{d}^{\widehat\nabla}A=0, \ \forall X\in\Gamma(TM)$ if and only if $C=\Re(q)$, where $q$ is a cubic holomorphic differential on $(M,J)$. If this is the case, then $q=C(\cdot,\cdot,\cdot)+iC(J\cdot,J\cdot,J\cdot)$
\end{theorem}
\begin{remark}
Thanks to Corollary \ref{picktensoraffinesphere} we notice that if $M$ is an immersed affine $2$-sphere its Pick tensor is totally symmetric, hence it can always be expressed as the real part of a cubic holomorphic differential on $(M,J)$.
\end{remark}

\subsection{Wang's equation}
It is clear, from the previous section, that there is a deep relation between a fixed conformal structure on an affine $2$-sphere and the Pick tensor. Here, we want to give a more explicit treatment of this phenomenon following the work of Wang, see \cite{changping1990some}, and Loftin-McIntosh, see \cite{loftin2013cubic}.\newline
For convex affine spheres, the affine metric $h$ is positive-definite, and thus it provides a conformal structure in dimension two. In this case, the Pick tensor can be identified with the real part of a holomorphic cubic differential. We now derive the structure equations for these affine spheres in terms of the conformal structure. Choose $z=x+iy$ a local conformal coordinate with respect to $h$, so that $h=e^{\psi}\vl\mathrm{d}z\vl^2$, where $\vl\mathrm{d}z\vl^2$ is defined as the symmetric product between $\mathrm{d}z$ and $\mathrm{d}\bar{z}$. Parametrize the affine sphere by $f\!:\Delta \to \R^3$, where $\Delta$ is a simply-connected domain in $\C$. Since $\{e^{-\frac{1}{2}\psi}f_x, e^{-\frac{1}{2}\psi}f_y\}$ is a $h$-orthonormal basis of the tangent space, the affine normal satisfies $$\det (e^{-\frac{1}{2}\psi}f_x, e^{-\frac{1}{2}\psi}f_y, \xi) = 1$$ which implies $$\det (f_x, f_y, \xi) = e^\psi \ .$$ By rewriting all in terms of $$\frac{\partial f}{\partial z}=\frac{1}{2}(f_x-if_y) \quad \text{and} \quad \frac{\partial f}{\partial\bar{z}}=\frac{1}{2}(f_x+if_y)$$ we get $$\det(f_z, f_{\bar{z}}, \xi)=ie^\psi \ .$$
The affine structure equations are
\begin{equation}\label{structurequationsaffinesphere}\begin{aligned}
D_XY &= \nabla_XY + h(X,Y)\xi \\
D_X\xi &= -\lambda\cdot X \ . 
\end{aligned}\end{equation}
Now consider the coordinate frame $\{e_1:=f_z:=f_*(\frac\partial{\partial z}), \ e_{\bar 1}:=f_{\bar z}:=f_*(\frac\partial{\partial \bar z})\}$. Hence,
$$ h(f_z,f_z)= h(f_{\bar z}, f_{\bar z}) = 0,\quad h(f_z,f_{\bar z}) = \frac{1}{2}e^{\psi}.$$
Let  $\theta$ be the matrix of connection one-forms for $\nabla$, i.e. 
$$ \nabla e_i = \theta^j_ie_j, \quad i,j\in\{1,\bar 1\}.$$
If $\hat\theta$ is the matrix of connection one-forms of the Levi-Civita connection, then $$ \hat\theta^1_{\bar 1} = \hat\theta^{\bar 1}_1 = 0, \quad \hat\theta^1_1 = \partial \psi, \quad \hat \theta^{\bar 1}_{\bar 1} = \bar\partial \psi \ . $$ The difference $\nabla-\widehat\nabla$ is equal to the so-called \emph{Pick form}, namely the section of $T^*M\otimes\End(TM)$ satisfying Proposition \ref{prop:pickform}. In local coordinates $$\theta_i^j-\hat\theta_i^j=A^j_{ik}\rho^k,\quad i,j\in\{1,\bar{1}\}$$ where $\{\rho^1=\mathrm{d}z, \rho^{\bar{1}}=\mathrm{d}\bar{z}\}$ is the dual frame of one-forms. By lowering an index we get the Pick tensor $$C_{ijk}=h_{il}A^l_{jk},\quad i,j,k\in\{1,\bar{1}\}$$ which is totally symmetric, as one can see from the last equation. In particular, all the components of $C$ must vanish except for $C_{111}$ and $\overline{C_{111}}=C_{\bar{1}\bar{1}\bar{1}}$. This discussion completely determines $\theta$, indeed
\begin{equation*}
   \theta=\begin{pmatrix}\theta_1^1 & \theta^1_{\bar{1}} \\ \theta_1^{\bar{1}} & \theta_{\bar{1}}^{\bar{1}}
    \end{pmatrix}=\begin{pmatrix}\partial\psi & e^{-\psi}\bar{Q}\mathrm{d}\bar{z} \\ e^{-\psi}Q\mathrm{d}z & \bar\partial\psi
    \end{pmatrix}
\end{equation*}where $Q:=2C_{111}$ is a smooth function on the affine sphere.\newline Since $D$ is the standard (flat) connection on $\R^3$, by using the structure equations (\ref{structurequationsaffinesphere}) we have
\begin{equation}\label{structurequationsconformal}\begin{aligned}
f_{zz}&:=D_{f_z}f_z=\nabla_{\frac{\partial}{\partial z}}\frac{\partial}{\partial z}=\psi_zf_z+e^{-\psi}Qf_{\bar z} \\
f_{\bar z\bar z}&:=D_{f_{\bar z}}f_{\bar z}=\nabla_{\frac{\partial}{\partial\bar z}}\frac{\partial}{\partial\bar z}=\psi_{\bar z}f_{\bar z}+e^{-\psi}\bar Qf_z \\ f_{z\bar z}&:=D_{f_z}f_{\bar z}=\nabla_{\frac{\partial}{\partial z}}\frac{\partial}{\partial\bar z}+h(f_z,f_{\bar z})\xi=\frac{1}{2}e^\psi\xi \ .
\end{aligned}\end{equation}We can translate our affine sphere so that $\xi=-\lambda f$, hence by combining (\ref{structurequationsconformal}) with this last equation we get a $1^{st}$-order linear system in $\big(\mathbf{F}_\lambda\big)^t:=(f_z, f_{\bar z}. \xi)$, given by \begin{equation}\label{ODEsystem}\begin{aligned}
\frac\partial{\partial z}\begin{pmatrix}
f_z \\ f_{\bar z} \\ \xi \end{pmatrix}&=\begin{pmatrix} \psi_z & Qe^{-\psi} &0 \\
 0 & 0 & \frac{1}{2}e^{\psi}  \\ -\lambda & 0 & 0\end{pmatrix}\begin{pmatrix}f_z \\ f_{\bar z} \\ \xi \end{pmatrix} \\
\frac\partial{\partial \bar z}
\begin{pmatrix} f_z \\ f_{\bar z} \\ \xi \end{pmatrix}&=
\begin{pmatrix} 0 & 0 & \frac{1}{2}e^{\psi}\\
 \bar Q e^{-\psi} & \psi_{\bar z} & 0 \\ 0 & -\lambda & 0 \end{pmatrix}
\begin{pmatrix} f_z \\ f_{\bar z} \\ \xi \end{pmatrix} \ . 
\end{aligned}\end{equation}Given an initial condition for $\big(\mathbf{F}_\lambda\big)^t$ at $z_0\in\Delta$, there exists a unique solution to this system as long as the following integrability conditions\footnote{They can be traced back to the Frobenius Theorem.} are satisfied $$(f_{zz})_{\bar z}=(f_{z\bar z})_z,\quad (f_{\bar z\bar z})_z=(f_{\bar z z})_{\bar z} \ . $$ By using (\ref{structurequationsconformal}), the above equations can be rewritten as
    \begin{equation}\label{integrabilityequations}\begin{aligned}&\psi_{z\bar z}+\vl Q\vl^2e^{-2\psi}+\frac{\lambda}{2}e^\psi=0 \\
        &Q_{\bar z}=0
    \end{aligned}\end{equation}The second equation and the definition of $Q$ implies that $q:=Q\mathrm{d}z^3$ is a holomorphic cubic differential over $\Delta$. Thus, we get the following
    \begin{theorem}[Integrability Theorem]\label{integrabilitytheorem}
    Fix $\lambda\in\{-1,-1\}$. Let $\Delta\subset\C$ be a simply-connected domain, $q=Q\mathrm{d}z^3$ a cubic holomorphic differential on $\Delta$, $\psi:\! \Delta\to\R$ a function satisfying $$\psi_{z\bar z}+\vl Q\vl^2e^{-2\psi}+\frac{\lambda}{2}e^\psi=0$$ $z_0\in\Delta$ and $\xi_0,f_0\in\R^3, a\in\C^3$ so that $\det(a,\bar a,\xi_0)=ie^{\psi(z_0)}$. Then, there exists a unique immersion of affine $2$-sphere $f:\!\Delta\to\R^3$ so that $$f(z_0)=f_0,\quad\xi(z_0)=\xi_0,\quad f_z(z_0)=a,\quad f_{\bar z}(z_0)=\bar a \ . $$Moreover, the pull-back under $f$ of the Blaschke metric and the Pick tensor are $e^\psi\vl\mathrm{d}z\vl^2$ and $Q\mathrm{d}z^3$ respectively.
    \end{theorem}
    \begin{remark}
    It is clear from this last theorem that in order to find examples of complete hyperbolic affine $2$-spheres we need only to find examples of triple $(\psi, q, \lambda=-1)$ on a simply-connected domain $\Delta\subset\C$ satisfying (\ref{integrabilityequations}).
    \end{remark}
    The strategy to produce such examples was outlined by Wang in \cite[\S 3]{changping1990some} following a work on constant mean curvature surfaces in pseudo-Riemannian space forms by B. Palmer, see \cite{palmer1990spacelike}.\newline Let $\Sg$ be a closed Riemann surface with genus $g\ge 1$. By the well-known Poincar\'e-Koebe Uniformization Theorem we can pick a Riemannian metric $g_0$ of constant curvature $k_0$ on $\Sg$ which is compatible with the initial complex structure. Let $H^0(\Sg, K^3)$ be the holomorphic sections of the tri-canonical bundle over $\Sg$, namely the $\C$-vector space of holomorphic cubic differentials over $\Sg$. It is easy to see, using the Riemann-Roch Theorem, that this space has complex dimension equal to $5g-5$, when $g\ge 2$, and equal to one, when $g=1$. If $z=x+iy$ is a local holomorphic coordinate for $\Sg$, then we can define a norm on $H^0(\Sg, K^3)$, given by: $$\vl\vl q\vl\vl_0^2:=\vl Q\vl^2e^{-3\phi}$$where $q=Q\mathrm{d}z^3$ and $g_0=e^\phi\vl\mathrm{d}z\vl^2$ in this local coordinate. \begin{proposition}\label{semilinearellipticproposition}Let $h=e^\psi\vl\mathrm{d}z\vl^2$ be a Hermitian metric in the same conformal class of $g_0$ and $q\in H^0(\Sg,K^3)$, then $(\psi,q,-1)$ satisfies (\ref{integrabilityequations}) if and only if the function $u:\!\Sg\to\R$, given by $h=e^u g_0$, satisfies:\begin{equation}\label{semilinearelliptic}
        \Delta_{g_0}u+4\vl\vl q\vl\vl^2_0e^{-2u}-2e^u-2k_0=0
    \end{equation}where $\Delta_{g_0}=4e^{-\phi}\partial_z\partial_{\bar z}$ is the Laplace-Beltrami operator on $\Sg$ with respect to $g_0$. 
    \end{proposition}
This result implies that the problem of finding a complete hyperbolic affine $2$-sphere can be rephrased in terms of a semi-linear elliptic equation defined on our Riemann surface $\Sg$. This last equation appears in many different contexts, based on the signs that are placed in front of its terms, see \cite{loftin2013cubic}. Ours is the case in which, thanks to the arrangement of the signs, we can prove many properties of this equation. In particular, the following holds: \begin{lemma}\label{loftinlemma}Let $(M, \tilde{g})$ be a Riemannian compact smooth manifold and let $\tilde{h}$ be a smooth non-negative function on $M$. Then, the equation \begin{equation}\label{loftinequation}
    \Delta_{\tilde g}u+\tilde h(p)e^{-2u}-2e^u+2=0
\end{equation}has a unique smooth solution.
\end{lemma}
The proof of this general result is due to Loftin and it simplifies a lot the original one given by Wang, using the theory of elliptic operators between Sobolev Spaces. It must be noted that Lemma \ref{loftinlemma} can be applied to (\ref{semilinearelliptic}) only when we can choose the initial metric $g_0$ so that $k_0=-1$ and this, clearly, can be done if $g\ge 2$. In the case of genus one the proof of existence and uniqueness of the solution for (\ref{semilinearelliptic}) is even simpler and will be covered in Section 2.3 below. At this point, it is clear that we can always solve (\ref{integrabilityequations}) and construct a complete hyperbolic affine $2$-sphere. Let $\pi:\widetilde \Sg\to\Sg$ be the conformal universal covering.\footnote{Notice that $\widetilde{\Sg}$ can always be indentified with a simply-connected domain in $\C$} Given a non-zero holomorphic cubic differential $q$ on $\Sg$, we get by Lemma \ref{loftinlemma} a unique triple $(h, q, -1)$ satisfying (\ref{semilinearelliptic}) on $\Sg$, where $h$ is the Hermitian metric given by $e^{u+\phi}\vl\mathrm{d}z\vl^2$ in local coordinates. Then, by Proposition \ref{semilinearellipticproposition} the triple $(\pi^*h, \pi^*q, -1)$ satisfies (\ref{integrabilityequations}), where $\pi^*h=e^\psi\vl\mathrm{d}z\vl^2$ on $\widetilde \Sg$. Since $\widetilde{\Sg}$ is biholomorphic either to $\C$ or to the open unit disk, by Theorem \ref{integrabilitytheorem} we know that $(\pi^*h, \pi^*q, -1)$ determines a hyperbolic affine $2$-sphere $f:\!\widetilde \Sg\to\R^3$ with $\pi^*h$ as its Blaschke metric and it is complete since $\pi:\!(\widetilde{\Sg}, \pi^*h)\to(\Sg, h)$ is a local isometry and $\Sg$ is compact. Moreover, it can be proved that the deck transformation group of $\Sg$ can be regarded as a discrete subgroup of the unimodular affine group acting on the affine sphere $f:\widetilde \Sg\to\R^3$. This holds because given any $\gamma\in\pi_1(\Sg)$ we have $(\gamma^*\circ\pi^*)h=\pi^*h$ and $(\gamma^*\circ\pi^*)q=\pi^*q$, but the Blaschke metric and the Pick form completely determine the affine sphere up to unimodular affine transformations. Hence, the map which sends the point $f(p)$ to $f(\gamma(p))$, with $p\in \Sg$, must be the restriction of an unimodular affine transformation in $\R^3$. By the standard theory of affine differential geometry it follows that the given construction yields all complete hyperbolic affine $2$-spheres which admit the action of a discrete subgroup of the unimodular affine group in $\R^3$ with compact quotient.

\subsection{The genus one case}

In this section we aim to study equation (\ref{semilinearelliptic}) in the case of genus one. We will prove, by a simple argument, that there exists a unique smooth solution, in fact constant. Moreover, we will show that any properly convex $\RP$-structure over a torus is covered by a triangle. In the following we will use the same notation as in Section 2.2.\newline If the Riemann surface $\Sg=T^{2}$ has genus one, we can pick the initial  metric $g_0$ to be flat, thus it can be written as $g_0=\vl\mathrm{d}z\vl^2$. A holomorphic cubic differential $q$ on $T^2$ is given (globally) by $q=c\mathrm{d}z^3$, with $c\in \C$, hence in this case equation (\ref{semilinearelliptic}) is \begin{equation}\label{semilinearelliptictorus}
\Delta_0u+4\vl c\vl^2e^{-2u}-2e^u=0
\end{equation}
where $u$ is the conformal parameter of the new metric $g$ and $\Delta_0=4\partial_z\partial_{\bar z}$ is the standard Laplacian. Notice that if the holomorphic cubic differential is zero, namely if $c=0$, then we get $\Delta_0u=2e^u$ and by integrating with respect to the volume form of $g_0$ it follows $$\int_{T^2}\Delta_0u \ \mathrm{d} \mu_0=2\int_{T^2}e^u \ \mathrm{d}\mu_0$$ which is not possible since the left hand side of the equation is zero and the right one is strictly positive.
\begin{proposition} \label{prop:sol_torus}
Provided $c\neq 0$, Equation (\ref{semilinearelliptictorus}) has a unique constant solution given by $u=$\sol
\end{proposition}
\begin{proof}
It is straightforward to see that \sol satisfies (\ref{semilinearelliptictorus}). Suppose $u$ is any other solution and let $p\in T^2$ be a maximum point for $u$. This implies that $(\Delta_0u)(p)\le 0$, hence $$2e^{u(p)}\le 4\vl c\vl^2e^{-2u(p)} \quad\Longrightarrow\quad u(p)\le\text{\sol}$$ Since $p$ is a point of maximum, we get $$u(x)\le u(p)\le\text{\sol},\quad\forall x\in T^2$$
Arguing in the same way with a point of minimum, we get $u(x)\ge\text{\sol},\forall x\in T^2$ thus the only possibility is that $$u\equiv\text{\sol} \ . $$
\end{proof}
We can already notice a first difference with the case of higher genus, in which the solution to the semi-linear elliptic equation (\ref{semilinearelliptic}) could always be found. On the torus, if on the one hand we have to place restrictions on the possible values of the cubic holomorphic differential, on the other the treatment is considerably simplified. In this case, the function $\psi$ of Theorem $\ref{integrabilitytheorem}$ coincides with the unique solution of $(\ref{semilinearelliptictorus})$. In particular, following the argument explained at the beginning of the previous section, we can rewrite the first order system of ODEs (\ref{ODEsystem}) in the following way

\begin{equation}
\begin{split}
\label{ODEsystemtorus}\begin{aligned}
\frac\partial{\partial z}\begin{pmatrix}
f_z \\ f_{\bar z} \\ f \end{pmatrix}&=\begin{pmatrix} 0 & ce^{-\psi} &0 \\
 0 & 0 & \frac{1}{2}e^{\psi}  \\ 1 & 0 & 0\end{pmatrix}\begin{pmatrix}f_z \\ f_{\bar z} \\ f \end{pmatrix} \\
\frac\partial{\partial \bar z}
\begin{pmatrix} f_z \\ f_{\bar z} \\ f \end{pmatrix}&=
\begin{pmatrix} 0 & 0 & \frac{1}{2}e^{\psi}\\
 \bar c e^{-\psi} & 0 & 0 \\ 0 & 1 & 0 \end{pmatrix}
\begin{pmatrix} f_z \\ f_{\bar z} \\ f \end{pmatrix} \ . 
\end{aligned}
\end{split}
\end{equation}

In a more compact form if $\mathbf{F}^t=(f_z, f_{\bar z}, f)$ and $A,B$ are the $3\times 3$ matrices in the first and second equation respectively, we get $$\begin{cases} \frac\partial{\partial z}\mathbf{F}=A\cdot\mathbf{F} \\ \frac\partial{\partial\bar z}\mathbf{F}=B\cdot\mathbf{F} \ .
\end{cases}$$
As one can easily check, a solution of this system is given by \begin{equation}\label{solutionsystemtorus}\mathbf{F}(z,\bar z)=e^{Az+B\bar z}\cdot C\end{equation}where $C$ is a constant matrix determined by the initial data (see Theorem \ref{integrabilitytheorem}). In order to show that every properly convex $\RP$-structure on $T^2$ is covered by a triangle, we need to compute an explicit solution of (\ref{ODEsystemtorus}), that is to find a formula for the parametrization $f:\C\to\R^3$ of the hyperbolic affine $2$-sphere. Then, using Theorem \ref{thmchengyau} we get that this affine sphere is asymptotic to a cone over a bounded domain $\Omega$, which will turn out to be projectively equivalent to a triangle. The main result of this section is the following:
\begin{theorem}
Let $\Omega/_{\Gamma}$ be a properly convex $\RP$-structure on $T^2$, then $\Omega$ is projectively equivalent to a triangle in $\R^3$ with vertices $\{(1,0,0); (0,1,0); (0,0,1)\}$.
\end{theorem}
\begin{proof}
Recall that the holomorphic cubic differential is given by $q=c\mathrm{d}z^3$, with $c\neq0$ and $c=\rho e^{i\theta}$ with $\rho>0$ and $\theta\in\R$. Since to find the solution $\mathbf{F}$ of (\ref{ODEsystemtorus}) we have to compute the exponential of a sum of matrices, the first step is to observe that $Az$ and $B\bar z$ commute. Moreover, we can find a common basis of eigenvectors that diagonalizes them simultaneously, namely there exists an invertible matrix $P$ such that $$Az=PD_{Az}P^{-1}\quad\text{and}\quad B\bar z=PD_{B\bar z}P^{-1}$$ with $D_{Az}, D_{B\bar z}$ diagonal matrices. From this it follows that $$Az+B\bar z=PD_{Az+B\bar z}P^{-1}\quad\Longrightarrow\quad e^{Az+B\bar z}=Pe^{D_{Az+B\bar z}}P^{-1} \ ,$$
where $D_{Az+B\bar z}=D_{Az}+D_{B\bar{z}}$.

A common basis of eigenvectors is given by $$\vec{v}_0=\begin{pmatrix}
e^{i\frac 2{3}\theta} \\ 1 \\ \big(\frac\rho{2}\big)^{-\frac 1{3}}e^{i\frac\theta{3}}\end{pmatrix},\quad\vec{v}_1=\begin{pmatrix}
\zeta^2e^{i\frac 2{3}\theta} \\ 1 \\ \zeta\big(\frac\rho{2}\big)^{-\frac 1{3}}e^{i\frac\theta{3}}\end{pmatrix},\quad \vec{v}_2=\begin{pmatrix}
\zeta e^{i\frac 2{3}\theta} \\ 1 \\ \zeta^2\big(\frac\rho{2}\big)^{-\frac 1{3}}e^{i\frac\theta{3}}\end{pmatrix}$$ with eigenvalues $\{\lambda_0z, \zeta\lambda_0 z, \zeta^2\lambda_0z\}$ for $Az$ and eigenvalues $\{\bar{\lambda}_0\bar z, \zeta^2\bar{\lambda}_0\bar z, \zeta\bar{\lambda}_0\bar z\}$ for $B\bar z$, where $\zeta=e^{i\frac{2\pi}{3}}$ is a 3rd primitive root of unity and $\lambda_0:=\big(\frac\rho{2}\big)^{\frac 1{3}}e^{i\frac\theta{3}}$. Hence, the matrix $P$ is given by $(\vec v_0 \ | \ \vec v_1 \ | \ \vec v_2)$ and the eigenvalues of $Az+B\bar z$ are $\{2\Ree(\lambda_0z), 2\Ree(\lambda_0z\zeta), 2\Ree(\lambda_0z\zeta^2)\}$. At this point, it is easy to compute the matrix $e^{Az+B\bar z}$ and find the parametrization $f$, being it the third row of the solution of the system $\mathbf{F}$. The vector $f=(f_1, f_2, f_3)^t$ we obtain takes values in $\C^3$ and not in $\R^3$ as one might expect. This happens because we still have to make a choice of the initial data. Then, by choosing the following constant matrix $C$ in (\ref{solutionsystemtorus}) 
\begin{equation}
    C=\begin{pmatrix}
    \big(\frac\rho{2}\big)^{\frac 1{3}}e^{i\frac\theta{3}} & \zeta\big(\frac\rho{2}\big)^{\frac 1{3}}e^{i\frac\theta{3}} & \zeta^2\big(\frac\rho{2}\big)^{\frac 1{3}}e^{i\frac\theta{3}} \\ \big(\frac\rho{2}\big)^{\frac 1{3}}e^{-i\frac\theta{3}} & \zeta^2\big(\frac\rho{2}\big)^{\frac 1{3}}e^{-i\frac\theta{3}} & \zeta\big(\frac\rho{2}\big)^{\frac 1{3}}e^{-i\frac\theta{3}} \\ 1 & 1 & 1
    \end{pmatrix}
\end{equation}
we get
\begin{equation}
    f(z,\bar z)=\begin{pmatrix}e^{2\Ree(\lambda_0z)} \\ e^{2\Ree(\zeta\lambda_0z)} \\ e^{2\Ree(\zeta^2\lambda_0z)}
    \end{pmatrix}\in\R^3 \ .
\end{equation}It is a straightforward computation to see that $$\Ree(\lambda_0z)+\Ree(\zeta\lambda_0z)+\Ree(\zeta^2\lambda_0z)=0 ,$$ hence showing that $f$ is a parametrization of the hypersurface $\{(x,y,w)\in\R^3 \ | \ xyw=1, \ x,y,w>0\}$. In particular, the hyperbolic affine $2$-sphere we get is asymptotic to the three coordinate planes in the first octant, which are nothing but the boundary of the cone over the triangle $T$ contained in the plane $\{(x,y,w)\in\R^3 \ | \ x+y+w=1\}$ and with vertices $\{(1,0,0); (0,1,0); (0,0,1)\}$. By Theorem \ref{thmchengyau}, this triangle has to be projectively equivalent to the convex bounded domain $\Omega=p\big(\mathcal{C}(T)\big)=\{[x:y:w]\in\RP \ | \ x,y,w>0\}$ of the initial properly convex $\RP$-structure, where $p:\R^3\setminus\{0\}\to\RP$ is the standard projection.
\end{proof}
\begin{remark}
It must be noted that in the case of genus one, the problem of classifying convex bounded domains $\Omega$, as in Definition \ref{defproperlyconvexstructure}, up to projective transformations and preserved by the action of a discrete subgroup $\Gamma<\SL(3,\R)$ contained in $\Proj(\Omega)$ and isomorphic to $\Z\times\Z$, is equivalent to the problem of classifying flat hyperbolic affine $2$-spheres up to unimodular affine transformations. The latter was the problem studied in \cite{magid1990flat} which we now recovered in terms of properly convex $\RP$-structures over the torus.
\end{remark}
\begin{corollary}\label{cor:cubictorusandproperlyconvex}
There exists a bijection between $\deft$ and the complement of the zero section in $\cubic$. 
\end{corollary}
\begin{proof}The above bijection follows from Theorem \ref{thmchengyau}. In fact, by Remark \ref{rmk:correspondencehypaffinesphereandpropconvex} to any $\big[\Omega/_{\Gamma}\big]\in\deft$ we have an associated equivariant hyperbolic affine $2$-sphere $M$ which is determined by its Blaschke metric and the Pick tensor (see Definition \ref{def:picktensor}). Hence, let $\chi:\deft\to Q^3_0(\Teich^c(T^2))$ be the map that associates to each $\big[\Omega/_{\Gamma}\big]$ the pair $(J,q)$ where $J$ is the complex structure induced by the Blaschke metric and $q=c\mathrm{d}z^3$ is a non-zero cubic holomorphic differential whose real part coincides with the Pick tensor of $M$. Since, by Proposition \ref{prop:sol_torus} and Theorem \ref{integrabilitytheorem}, for any such $(J,q)$ we can find a unique (up to unimodular affine transformations) hyperbolic affine $2$-sphere in $\R^{3}$ that is invariant under a subgroup $\Gamma < \SL(3,\R)$ isomorphic to $\pi_{1}(T^{2})$, the map $\chi$ is a bijection. 
\end{proof}
\section{The main results}
The purpose of this section is to present and prove the three results of this paper. First, we define a symplectic form and a complex structure on $\deft$ compatible with a pseudo-Riemannian metric. We then move on to show the existence of Hamiltonian actions of $S^1$ and $\SL(2,\R)$ on $\deft$ and we finally compute their moment maps.
\subsection{The pseudo-K\"ahler structure on \texorpdfstring{$\deft$}{B(T2)}}
We begin by defining the main objects we will use during this section. Let $\rho:=\mathrm{d}x_0\wedge\mathrm{d}y_0$ be the standard area form on $\R^2$ \begin{definition}
The space $\almost$ of $\rho$-compatible linear-complex structures on $\R^2$ is defined as
$$\almost:=\{J\in\End(\R^2) \ | \ J^2=-\mathds{1}, \ \rho(v,Jv)>0 \ \text{for some} \ v\in\R^2\setminus\{0\}\} \ . $$
\end{definition}This space is a $2$-dimensional manifold and it is easy to see that $\forall J\in\almost$, the tensor $g_J(\cdot,\cdot):=\rho(\cdot,J\cdot)$ is a Riemannian metric on $\R^2$, with respect to which $J$ is an orthogonal endomorphism. By differentiating the identity $J^2=-\mathds{1}$, it follows that $$T_J\almost=\{\dot{J}\in\End(\R^2) \ | \ J\dot J+\dot J J=0\} \ . $$ Equivalently, the space $T_J\almost$ can be identified with the trace-less and $g_J$-symmetric endomorphisms of $\R^2$. It carries a natural (almost) complex structure given by
\begin{align*}
\mathcal{I}:T_J&\almost\to T_J\almost \\ &\dot J\mapsto-J\dot J \ . 
\end{align*}
There is a natural scalar product defined on each tangent space 
$$\langle\dot J, \dot J'\rangle_J:=\frac 1{2}\tr(\dot J\dot J')$$for each $\dot J,\dot J'\in\almost$. It is easy to check that $\mathcal{I}$ preserves this scalar product.\begin{lemma}[{\cite[Lemma 3.11]{mazzoli2021parahyperkahler}}]\label{lem:teichmulleralmost}There is a diffeomorphism between $\almost$ and $\Teich^c(T^2)$, which is equivariant with respect to the action of $MCG(T^2)\cong\SL(2,\Z)$
\end{lemma}\begin{corollary}\label{cor:bijectioncubicpick}
The holomorphic vector bundle $\cubic$ can be identified with the following
\begin{equation}
    D^3(\almost):=\{(J, C)\in \almost\times S_3(\R^2) \ | \ C(J\cdot,J\cdot,J\cdot)=-C(J\cdot,\cdot,\cdot)\}
\end{equation}where $S_3(\R^2)$ is the space of totally-symmetric tri-linear forms on $\R^2$. Moreover if $(J,C)\in\pick$, then \begin{equation}
    C(J\cdot,\cdot,\cdot)=C(\cdot,J\cdot,\cdot)=C(\cdot,\cdot,J\cdot) \ .
\end{equation}
\end{corollary}
\begin{proof}
If $J\in\almost$ and $q$ is a cubic $J$-holomorphic differential, then $C=\Ree(q)$ is a totally-symmetric tri-linear form on $\R^2$ by Theorem \ref{thm:picktensor}. In particular, $$C(X,Y,Z)=g_J(A(X)Y,Z),\qquad\forall X,Y,Z\in\Gamma(T\R^2)$$ where $A\in\End(\R^2)\otimes T^*\R^2$ is the Pick form and it satisfies the properties in Proposition \ref{prop:pickform}. Hence, for all $X,Y,Z\in\Gamma(T\R^2)$, we have
\begin{align*}
    C(JX,JY,JZ)&=-g_J(JA(JX)Y,JZ) \tag{$A(JX)\in T_J\almost$} \\ &=-g_J(A(JX)Y,Z) \tag{$J \ \text{is} \ g_J-\text{orthogonal}$} \\ &=-C(JX,Y,Z) \ .
\end{align*}

We conclude that $(J, C=\Ree(q))\in D^3(\almost)$. Conversely, if $(J, C)\in D^3(\almost)$, then $q=C(\cdot,\cdot,\cdot)-iC(J\cdot,\cdot,\cdot)$ defines a cubic holomorphic differential by Theorem \ref{thm:picktensor}. Finally, \begin{align*}
    C(JX,Y,Z)&=g_J(A(JX)Y,Z) \\ &=-g_J(A(JX)JY,JZ)  \tag{$A(JX)\in T_J\almost$} \\ &=-g_J(A(JY)JX,JZ) \tag{Proposition \ref{prop:pickform}} \\&=g_J(JA(JY)X,JZ) \\&=g_J(A(X)JY,Z) \\&=C(X,JY,Z)
\end{align*}for all $X,Y,Z\in\Gamma(T\R^2)$. A similar computation shows that $C(\cdot,J\cdot,\cdot)=C(\cdot,\cdot,J\cdot)$.
\end{proof}
\begin{remark}
Notice that, thanks to relation (\ref{picktensorandpickform}), the space $\pick$ can be interpreted in terms of the Pick form $A$, namely it is formed by all possible pair $(J, A)$ with $J\in\almost$ and $A\in\End(\R^2)\otimes T^*\R^2$ such that:\begin{itemize}
    \item[(a)]$A(X)Y=A(Y)X, \quad\forall X,Y\in\Gamma(T\R^2)$;
    \item[(b)] the endomorphism $A(X)$ is $g_J$-symmetric and trace-less for all vector field $X$. In particular, $A(X)\in T_J\almost$
\end{itemize}We will make repetitive use of this correspondence, using the Pick tensor or the Pick form whichever is more convenient.
\end{remark}
It is clear that $\pick$ can be interpreted as a vector bundle over $\almost$, whose fiber at a point $J\in\almost$ is a two dimensional real vector space, denoted by $\pick_J$. Hence, we can introduce a scalar product on this vector space by \begin{equation}\label{def:scalarproductpick}
    \langle A,B\rangle_J:=\int_{T^2}\tr(A\wedge\ast B)
\end{equation}where $\ast$ is the Hodge-star operator with respect to $g_J$. We can give a more precise description of this scalar product: let $J_0\in\almost$ be the standard linear complex structure, namely \begin{equation*}J_0=\begin{pmatrix}
0 & -1 \\ 1 & 0
\end{pmatrix}\end{equation*} then, $g_{J_0}(\cdot, \cdot)$ is the Euclidean metric on $\R^2$ and the standard basis $\{\frac\partial{\partial x_0}, \frac\partial{\partial y_0}\}$ is such that $J_0(\frac\partial{\partial x_0})=\frac\partial{\partial y_0}, J_0(\frac\partial{\partial y_0})=-\frac\partial{\partial x_0}$. Similarly, for any other $J \in \almost$, we can find a $g_J$-orthonormal basis $\{e_1, e_2\}$ such that $Je_1=e_2, \ Je_2=-e_1$. Then, if $\{e_1^*, e_2^*\}$ is the dual basis and $A,B\in\pick_J$, we have $A=A_1e_1^*+A_2e_2^*, \ B=B_1e_1^*+B_2e_2^*$ for some $2\times 2$ matrices $A_{i}$ and $B_{i}$. Hence $\ast B=B_1e_2^*-B_2e_1^*$. In the end, the scalar product is given by $$\langle A, B\rangle_J=\int_{T^2}\tr(A_1B_1+A_2B_2)e_1^*\wedge e_2^* \ . $$ 
\begin{remark}
We will be assuming that the area of the torus for the flat metric $g_{J}$ is equal to $1$, in other words
$$\int_{T^2}e_1^*\wedge e_2^*=1 \ . $$
Indeed there is an equivalent description of $\Teich^c(T^2)$ as the space of isotopy classes of unit-area flat metrics on $T^2$. This can be seen thanks to the isomorphism $\Teich^c(T^2)\cong\almost$ presented in Lemma \ref{lem:teichmulleralmost}. The space $\almost$ can be interpreted as the space of all orientation preserving linear maps $\R^2\to\R^2$ up to rotation and/or dilation. This is equivalent to classify all possible marked lattices in $\R^2$ up to Euclidean isometries and homotheties. Since we can always, up to homotheties, choose a marked lattice of unit area, it follows we can always find a $J\in[J']\in\almost$ with the above property. For a more detailed discussion see \cite[\S 10.2]{farb2011primer}.
\end{remark}
\begin{lemma}\label{lem:tangentpickform}
Let $(J,A)\in\pick$. Then, an element $(\dot J,\dot A)\in T_J\almost\times(\End(\R^2)\otimes T^*\R^2)$ belongs to $T_{(J,A)}\pick$ if and only if \begin{equation}\label{tracedotA}
    \dot J\in T_J\almost,\quad\dot A_0\in\pick_J,\quad \tr\dot A(X)=\tr(\dot JJA(X))=\tr(JA(X)\dot J)
\end{equation}for each $X\in\Gamma(T\R^2)$, where $\dot A:=g_J^{-1}\dot C$ denotes the unique element $\dot A \in \End(\R^2)\otimes T^*\R^2$ such that $g_{J}(\dot A(X)Y,Z)=\dot C(X,Y,Z)$ for all $X,Y,Z \in \Gamma(T\R^{2})$ and $\dot A_0$ is its trace-free part.
\end{lemma}
\begin{proof}
First notice that \begin{equation}\label{dotmetric}
    \dot g_J=\rho(\cdot,\dot J\cdot)=-\rho(\cdot,J^2\dot J\cdot)=-g_J(\cdot,J\dot J\cdot) \ .
\end{equation}
Then, since $A(X)=(g_{J}^{-1}C)(X)$ and the endomorphisms $A(X)$ are trace-free for each vector field $X$, we get $$0=\tr(A(X))'=\tr((g_{J}^{-1}C)(X))'=-\tr(g_J^{-1}\dot g_J g_J^{-1}C(X))+\tr(\dot A(X)) \ . $$ 
Using equation (\ref{dotmetric}) we obtain $$\tr(\dot A(X))=-\tr(J\dot JA(X))=\tr(\dot JJA(X)) \ . $$
\end{proof}
The group $\SL(2,\R)$ acts on $\almost$ by conjugation and more generally on its tangent space by
$$(J,\dot J)\in T\almost, \quad P\cdot (J,\dot J):=(PJP^{-1}, P\dot JP^{-1})$$ with $P\in\SL(2,\R)$.
\begin{lemma}
There is an $\SL(2,\R)$ action on $\pick$ given by: $$P\cdot (J,A):=(PJP^{-1}, PA(P^{-1}\cdot)P^{-1})$$ where $P\in\SL(2,\R)$ and $A(P^{-1}\cdot)$ has to be interpreted as the action of $P^{-1}$ by pull-back on the one-form part of $A$.
\end{lemma}
\begin{proof}
In Section 1.4 we introduced an $\SL(2,\R)$ action on $\cubic$ given by $$P\cdot(J,q)=(PJP^{-1}, (P^{-1})^*q)$$with $P\in\SL(2,\R)$ and $(J,q)\in\cubic$. We need to understand how this action transforms under the bijection of Corollary \ref{cor:bijectioncubicpick}. The new cubic  holomorphic differential $(P^{-1})^*q$ corresponds to the new Pick tensor $\widetilde C=\Ree\big((P^{-1})^*q\big)=C(P^{-1}\cdot, P^{-1}\cdot, P^{-1}\cdot)$ which is given by \begin{align*}\widetilde C(X,Y,Z)&=C(P^{-1}X, P^{-1}Y, P^{-1}Z) \\ &=g_{J}(A\big(P^{-1}X\big)P^{-1}Y, P^{-1}Z) \\&=\rho(PA\big(P^{-1}X\big)P^{-1}Y, PJP^{-1}Z) \tag{$P\in\SL(2,\R)$} \\ &=g_{P\cdot J}(PA\big(P^{-1}X\big)P^{-1}Y, Z) \ . \end{align*}
Hence, the new Pick form $\widetilde A$ defined by $$\widetilde C(X,Y,Z)=g_{P\cdot J}(\widetilde A(X)Y, Z)$$ is exactly $\widetilde A=PA(P^{-1}\cdot)P^{-1} \ . $
\end{proof}
\begin{lemma}\label{lem:invarianza}
For every $P\in\SL(2,\R)$ and $J\in\almost$, we have
\begin{align*}
&\langle P\cdot\dot J,P\cdot\dot J'\rangle_{P\cdot J}=\langle\dot J,\dot J'\rangle_J \\ &\langle P\cdot A,P\cdot B\rangle_{P\cdot J}=\langle A,B\rangle_J\end{align*} where $\dot J,\dot J'\in T_J\almost$ and $A,B\in\pick_J$.
\end{lemma}
\begin{proof}
For the action on $T_J\almost$, we have
\begin{align*}
    \langle P\cdot\dot J,P\cdot\dot J'\rangle_{P\cdot J}&=\frac 1{2}\tr(P\dot J\dot J'P^{-1}) \\&=\frac 1{2}\tr(\dot J\dot J') \tag{\text{trace symmetry}} \\ &=\langle\dot J,\dot J'\rangle_J \ .
\end{align*}For the action on $\pick_J$, we have
\begin{align*}
    \langle P\cdot A,P\cdot B\rangle_{P\cdot J}&=\int_{T^2}\tr(PA_1B_1P^{-1}+PA_2B_2P^{-1})(P^{-1})^*(e_1^*\wedge e_2^*)\\&=\int_{T^2}\tr (P(A_1 B_1+A_2B_2)P^{-1})e_1^*\wedge e_2^* \tag{$P\in\SL(2,\R)$}\\&=\langle A,B\rangle_J \tag{\text{trace symmetry}} \ .
\end{align*}
\end{proof}
The $\SL(2,\R)$-action on $\pick$ can be differentiated, hence we get a linear isomorphism between $T_{(J,A)}\pick$ and $T_{P\cdot (J,A)}\pick$, which is given explicitly by
$$P\cdot (\dot J,\dot A)=(P\dot JP^{-1}, P\dot A(P^{-1}\cdot)P^{-1})\\$$ where $(\dot J,\dot A)\in T_{(J,A)}\pick$ and $P\in\SL(2,\R)$. Moreover, all the conditions in Lemma \ref{lem:tangentpickform} are $\SL(2,\R)$-invariant.\newline We can define a similar scalar product on pairs $\dot A,\dot B$ by $$\langle\dot A, \dot B\rangle_J:=\int_{T^2}\tr(\dot A\wedge\ast\dot B)$$which is $\SL(2,\R)$-invariant as well. \newline In the following we will denote with $\vl\vl\cdot\vl\vl_J=\vl\vl\cdot\vl\vl$ the norm induced by the scalar product $\langle\cdot,\cdot\rangle_J=\langle\cdot,\cdot\rangle$ and it will be clear from the context which one we are using. In order to simplify the notation we define $\normpick:=\frac{1}{4}\vl\vl A\vl\vl_J^2$. Finally, since $A$ and $\dot A$ are both elements of $\End(\R^2)\otimes T^*\R^2$ it makes sense to consider their trace and trace-free part, namely if $A=A_1e_1^*+A_2e_2^*$ and $\dot A=\dot A_1e_1^*+A_2e_2^*$, then: \begin{align*}&A_0=(A_1)_0e_1^*+(A_2)_0e_2^*, \qquad A_{\tr}=\frac 1{2}\tr(A_1)\mathds{1}e_1^*+\frac 1{2}\tr(A_2)\mathds{1}e_2^*, \\ &\dot A_0=(\dot A_1)_0e_1^*+(\dot A_2)_0e_2^*, \qquad \dot A_{\tr}=\frac 1{2}\tr(\dot A_1)\mathds{1}e_1^*+\frac 1{2}\tr(\dot A_2)\mathds{1}e_2^* \ . 
\end{align*} \newline Now we have all the definitions and results to define the pseudo-Riemannian metric on $\pick$.
\begin{definition}\label{def:pseudoriemannian}
Let $f:[0,+\infty)\to(-\infty,0]$ be a smooth function such that:\begin{itemize}
    \item[(i)] $f'(t)<0, \quad\forall t\ge 0$
    \item[(ii)]$\lim_{t\to+\infty}f(t)=-\infty\ . $
\end{itemize}Then, we define the following symmetric bi-linear form on $T_{(J,A)}\pick$\begin{equation}\begin{aligned}\label{pseudoriemannianmetric}
    \g_{(J,A)}\big((\dot J,\dot A); (\dot J',\dot A')\big):=\big(1-f(\normpick)\big)\langle\dot J,\dot J'\rangle&+\frac{f'(\normpick)}{3}\langle \dot A_0, \dot A_0'\rangle \\ &-\frac{f'(\normpick)}{6}\langle \dot A_{\tr}, \dot A_{\tr}'\rangle 
\end{aligned}\end{equation}and the endomorphism $\i$ of $T_{(J,A)}\pick$ \begin{equation}
    \i_{(J,A)}(\dot J, \dot A):=(-J\dot J, -\dot AJ-A\dot J)
\end{equation}where the products $\dot A J$ and $A\dot J$ have to be interpreted as a matrix multiplication.
\end{definition}Matching these two objects together we get the following form: $$\ome(\cdot,\cdot)=\g(\cdot,\i\cdot)$$which is given by:\begin{equation}\begin{aligned}\label{symplecticform}\ome_{(J,A)}\big((\dot J, \dot A); (\dot J', \dot A')\big)=\big(f(\normpick)-1\big)\langle\dot J, J\dot J\rangle&-\frac{f'(\normpick)}{6}\langle\dot A_{\tr}, \ast\dot A_{\tr}'\rangle \\ &-\frac{f'(\normpick)}{3}\langle\dot A_0, \dot A_0'J\rangle \ . 
\end{aligned}\end{equation}
\begin{remark}The symmetric tensor $\g$ and the form $\ome$ are defined only in terms of the various scalar products $\langle\cdot,\cdot\rangle$ and $\normpick$, hence by Lemma \ref{lem:invarianza} they are both $\SL(2, \R)$-invariant. In particular, the complex structure $\i$ is uniquely determined by the relation $\ome(\cdot,\cdot)=\g(\cdot, \i\cdot)$ once the form $\ome$ and the tensor $\g$ are given. In our case, this implies that $\i$ is $\SL(2,\R)$-invariant as well.\end{remark}
\begin{lemma}
For every $\dot J,\dot J'\in T_J\almost$ we have \begin{equation}\label{producttangentalmost}
\dot J\dot J'=\langle\dot J,\dot J'\rangle_J\mathds{1}-\langle J\dot J,\dot J'\rangle_JJ \ .  \end{equation}
\end{lemma}
\begin{proof}
    Notice that $$J\dot J\dot J'=-\dot J J\dot J'=\dot J\dot J'J$$ Therefore, the matrix $\dot J\dot J'$ commutes with $J$, but it is straightforward to see that this is equivalent to $\dot J\dot J'\in\Span_{\R}\{\mathds{1}, J\}$, hence the thesis.
\end{proof}
\begin{lemma}\label{lem:decompositioncomplexstructure}
Let $\{e_1,e_2\}$ be a $g_J$-orthonormal basis of $\R^2$ such that $Je_1=e_2$ and $Je_2=-e_1$, and let $\{e_1^*, e_2^*\}$ be its dual basis. Then, writing $A=A_1e_1^*+A_2e_2^*$ and $\dot A=\dot A_1e_1^*+\dot A_2e_2^*$ we get
\begin{itemize}
    \item [(1)] $JA_2=A_1$
    \item[(2)] $-\dot A J-A\dot J=\underbrace{-(\dot A_1)_0Je_1^*-(\dot A_2)_0Je_2^*}_{\text{trace-less part}}\underbrace{-\frac {1}{2}\tr(\dot A_2)\mathds{1}e_1^*+\frac {1}{2}\tr(\dot A_1)\mathds{1}e_2^*}_{\text{trace part}} \ . $
\end{itemize}
\end{lemma}
\begin{proof}
(1) By definition $A_i=A(e_i)$ for $i=1,2$ and the vector $A_i\cdot e_j$ can be written as a linear combination of $e_1,e_2$. Then, it is sufficient to prove that $A_1\cdot e_i=JA_2\cdot e_i$, for $i=1,2$. Hence, if $$A_1\cdot e_1=\alpha_{11}e_1+\beta_{11}e_2,\qquad A_2\cdot e_1=\alpha_{21}e_1+\beta_{21}e_2$$we get $JA_2\cdot e_1=\alpha_{21}e_2-\beta_{21}e_1$, but since the basis $\{e_1,e_2\}$ is $g_J$-orthonormal we can write $$\beta_{11}=g_J(A_1\cdot e_1, e_2)=C(e_1, e_1, Je_1)$$ and \begin{align*}\alpha_{21}&=g_J(JA_2\cdot e_1, e_1) \\ &=g_J(A(e_2)\cdot e_1,e_1)\tag{$J \ \text{is} \ g_J-\text{orthogonal}$}\\ &=C(Je_1, e_1, e_1) \\ &=\beta_{11} \tag{$C(\cdot,\cdot,\cdot) \ \text{is totally-symmetric}$}
\end{align*}With the same argument one can prove that $-\beta_{21}=\alpha_{11}$ and that $A_1\cdot e_2=JA_2\cdot e_2$, obtaining the thesis. \newline (2) By using the decomposition $\dot A=\dot A_0+(\dot A)_{\tr}$ we get
\begin{equation}
    -\dot AJ=-(\dot A_1)_0Je_1^*-(\dot A_2)_0Je_2^*-\frac 1{2}\tr(\dot A_1)Je_1^*-\frac 1{2}\tr(\dot A_2)Je_2^* \ .
\end{equation} The same happens for the tensor $A$, hence, using Equation  (\ref{producttangentalmost}), \begin{align*}-A\dot J&=-\frac 1{2}\tr(A_1\dot J)\mathds{1}e_1^*-\frac 1{2}\tr(A_2\dot J)\mathds{1}e_2^*+\frac 1{2}\tr(JA_1\dot J)Je_1^*+\frac{1}2\tr(JA_2\dot J)Je_2^* \\ &=-\frac 1{2}\tr(\dot A_2)\mathds{1}e_1^*+\frac 1{2}\tr(\dot A_1)\dsone e_2^*+\frac 1{2}\tr(\dot A_1)Je_1^*+\frac{1}2\tr(\dot A_2)Je_2^*\end{align*}where in the last equality we used (\ref{tracedotA}) and $JA_2=A_1$. It is now clear that adding the two terms $-\dot AJ$ and $-A\dot J$ we get the desired formula in the statement.
\end{proof}
\begin{theorem}\label{thm:pseudokahlerpick}
The triple $(\g,\i,\ome)$ defines an $\SL(2,\R)$-invariant pseudo-K\"ahler structure on $\pick$.
\end{theorem}
\begin{proof}In order not to overload the following proof too much, the closedness of $\ome$ and the non-degeneracy of $\g$ are postponed to Lemma \ref{lem:nondegenerateclosed} at the end of the section, as it requires a computation in local coordinates.\newline
    $\bullet$ \underline{\emph{$\i^2=-\dsone$ and it is integrable.}}\newline The first claim is a calculation:\begin{align*}
        \i_{(J,A)}^2(\dot J, \dot A)&=\i_{(J,A)}\big(-J\dot J, -\dot AJ-A\dot J\big) \\ &=(J^2\dot J, -(-\dot AJ-A\dot J)J+AJ\dot J) \\ &=(-\dot J, -\dot A +A\dot J J+AJ\dot J) \\ &=(-\dot J, -\dot A)\tag{$\dot J J=-J\dot J$}
    \end{align*}For the second one, it is sufficient to prove that, under the bijection in Corollary \ref{cor:bijectioncubicpick}, the almost-complex structure $\i$ on $\pick$ corresponds to the multiplication by $-i$ on $\cubic$ (see Lemma \ref{lem:Trautweinfibremap}, part $(2)$). Since the latter is integrable, the former is integrable as well. To show this, we need to compute the Pick tensor $\widetilde C$ associated with the variation $-i\dot q$ of the cubic differential $q$ in the fibre over $J$. Thanks to Corollary \ref{cor:bijectioncubicpick} this is given by $$\widetilde C(\cdot,\cdot,\cdot)=\Ree(-i\dot q)=-\dot C(\cdot,J\cdot,\cdot)-C(\cdot,\dot J\cdot,\cdot)\ .$$
    The relation with the new Pick form $\widetilde A$ is the one in Definition \ref{def:picktensor} and since \begin{align*}
        \widetilde C(X,Y,Z)&=-\dot C(X,JY,Z)-C(X,\dot JY,Z) \\ &=g_J\big((-\dot AJ-A\dot J)(X)Y,Z\big)
    \end{align*}for all $X,Y,Z\in\Gamma(T\R^2)$, we get that $\widetilde A=-\dot AJ-A\dot J$, hence the thesis. \newline
    $\bullet$ \underline{\emph{The metric $\g$ and the complex structure $\i$ are compatible.}} \newline We need to prove that $$\g_{(J,A)}\big(\i_{(J,A)}(\dot J,\dot A); \ \i_{(J,A)}(\dot J', \dot A')\big)=\g_{(J,A)}\big((\dot J, \dot A); (\dot J',\dot A')\big) \ . $$ By definition of $\i$ we have \begin{align*}\g_{(J,A)}\big(\i_{(J,A)}(\dot J,\dot A); \ \i_{(J,A)}(\dot J', \dot A')\big)&=\big(1-f(\normpick)\big)\langle-J\dot J,-J\dot J'\rangle \\ &+\frac{f'(\normpick)}{3}\langle -(\dot AJ+A\dot J)_0, -(\dot A'J'+A'\dot J')_0\rangle \\ &-\frac{f'(\normpick)}{6}\langle-(\dot AJ+A\dot J)_{\tr}, -(\dot A'J'+A'\dot J')_{\tr}\rangle \ . 
\end{align*}Since the argument of the functions $f, f'$ depends only on the norm of $A$ (up to a constant) it has not been changed when we applied $\i_{(J,A)}$, then we can focus only on the scalar products part. The first term is \begin{align*}\langle-J\dot J,-J\dot J'\rangle&=\frac 1{2}\tr(J\dot JJ\dot J') \\ &=\frac 1{2}\tr(\dot J\dot J') \tag{$J\dot J=-\dot JJ$} \\ &=\langle\dot J,\dot J'\rangle \ . \end{align*}
Applying part (2) of Lemma \ref{lem:decompositioncomplexstructure} and observing that $(\dot A_i)_0,(\dot A_i')_0\in T_J\almost, \ i=1,2$, the second term is \begin{align*}
   \langle -(\dot AJ+A\dot J)_0, -(\dot A'J'+A'\dot J')_0\rangle&=\int_{T^2}\bigg(\tr((\dot A_1)_0J(\dot A_1')_0J)+\tr((\dot A_2)_0J(\dot A_2')_0J)\bigg)e_1^*\wedge e_2^* \\ &=\int_{T^2}\bigg(\tr((\dot A_1)_0(\dot A_1')_0+(\dot A_2)_0(\dot A_2')_0)\bigg)e_1^*\wedge e_2^* \\ &=\langle \dot A_0,\dot A_0'\rangle \ . 
\end{align*}
Applying part (2) of Lemma \ref{lem:decompositioncomplexstructure} the third term is: \begin{align*}
    \langle-(\dot AJ+A\dot J)_{\tr}, -(\dot A'J'+A'\dot J')_{\tr}\rangle&=\frac 1{4}\int_{T^2}\tr(\tr(\dot A_1)\tr(\dot A_1')\dsone+\tr(\dot A_2)\tr(\dot A_2')\dsone)e_1^*\wedge e_2^* \\&=\langle \dot A_{\tr},\dot A_{tr}'\rangle \ .
\end{align*}Hence, we have the thesis.
\end{proof}

\begin{remark}\label{rmk:complexstructureoncomplementzerosection}
The complex structure $\i$ preserves the $0$-section of $T_{(J,A)}\pick$ since $\i_{(J,0)}(\dot J,0)=(-J\dot J, 0)$. In particular, $\i$ still defines a complex structure on the complement of the $0$-section on $\pick$, which is identified with $Q^3_0(\Teich^c(T^2))$ by Corollary \ref{cor:bijectioncubicpick} which is further identified with $\deft$ by Corollary \ref{cor:cubictorusandproperlyconvex}. Hence, we get a well-defined complex structure on $\deft$ which will be denoted with $\i$ by abuse of notation. The same argument holds for the pseudo-Riemannian metric $\g$ and the symplectic form $\ome$.
\end{remark}
\begin{manualtheorem}A
The deformation space $\deft$ admits a $MCG(T^2)$-equivariant pseudo-K\"ahler structure $(\g,\i,\ome)$.
\end{manualtheorem}

\begin{proof}
By Theorem \ref{thm:pseudokahlerpick} and Remark \ref{rmk:complexstructureoncomplementzerosection} the deformation space $\deft$ has a well-defined pseudo-K\"ahler structure $(\g,\i,\ome)$. Since all the identifications are equivariant with respect to $\SL(2,\Z)\cong MCG(T^2)$ and the triple $(\g,\i,\ome)$ is $\SL(2,\R)$-invariant, it follows that the induced pseudo-K\"ahler structure is $MCG(T^2)$-invariant.
\end{proof}
\begin{remark}
In the case $g\ge 2$, we will no longer have a family of pseudo-Kahler structures on $\defg$, but we will have to choose a particular function $f$ in order to find a triple $(\g,\ome,\i)$. The precise expression of the function is the following
\begin{equation}\label{functionfhighergenus}
    f(t)=-\bigg(\int_0^tF'(s)s^{-\frac 1{3}}\bigg)t^{\frac 1{3}}
\end{equation}where $F:[0,+\infty)\to\R$ is the unique smooth function such that
$$ce^{-F(t)}-2te^{-3F(t)}+1=0$$ with $c:=\frac{2\pi(2-g)}{\Vol(\Sg,\rho)}.$ 
It is not difficult to see that such function $f$ satisfies all the properties required in Theorem \ref{thmA}. 
\end{remark}


It only remains to prove that the symmetric tensor $\g$ and the $2$-form $\ome$ on $\pick$ are non-degenerate and closed, respectively. As we stated above, we need to write their expression in local coordinates. First of all, it is necessary to find the analogue in coordinates of the two spaces, $\almost$ and $\pick$, which we have studied so far. Let $G$ and $\Omega$ be the restriction of $\g$ and $\ome$ to the $0$-section of $\pick$, which is identified with $\almost$. Then, $$G_J(\dot J, \dot J')=\langle\dot J,\dot J'\rangle_J,\qquad\qquad\Omega_J(\dot J,\dot J')=-\langle \dot J,J\dot J'\rangle_J$$ with $\dot J,\dot J'\in T_J\almost$. In this case $G_J$ is a scalar product for all $J\in\almost$, hence $(G,\Omega)$ is an $\SL(2,\R)$-invariant K\"ahler structure on $\almost$. Moreover, the $\SL(2,\R)$-action is transitive with stabilizer SO$(2)$ at the standard linear complex structure $$J_0=\begin{pmatrix}0 & -1 \\ 1 & 0\end{pmatrix} \ . $$ Therefore, $\almost\cong\SL(2,\R)/$SO$(2) \cong \Hyp$. 
\begin{lemma}[{\cite[Lemma 4.3.2]{trautwein2018infinite}}]
Let $\Hyp$ be the hyperbolic plane with complex coordinate $z=x+iy$ and with K\"ahler structure $$g_{\Hyp}=\frac{\mathrm{d}x^2+\mathrm{d}y^2}{y^2} \qquad\qquad \omega_{\Hyp}=-\frac{\mathrm{d}x\wedge\mathrm{d}y}{y^2} \ . $$ Then, there exists a unique $\SL(2,\R)$-invariant K\"ahler isometry $j:\Hyp\to\almost$ such that $j(i)=J_0$. It is given by the formula \begin{equation}\label{kahlerisometryalmost}
    j(x+iy):=\begin{pmatrix}
    \frac{x}{y} & -\frac{x^2+y^2}{y} \\ \frac 1{y} & -\frac x{y} 
    \end{pmatrix} \ . 
\end{equation}
\end{lemma}
\begin{remark}
The minus sign in front of the area form on $\Hyp$ shows up since we are considering the relation $\ome(\cdot,\cdot)=\g(\cdot,\i\cdot)$ on $\pick$, hence on $\almost$.
\end{remark}
In particular, thanks to this last lemma and the isomorphism $\Teich^c(T^2)\cong\almost$, we can identify the Teichm\"uller space of the torus with $\Hyp$. Whenever we are thinking of the Teichm\"uller space of the torus as $\Hyp$, we will denote the total space of $\cubic$ as $\cubichyp$. In particular, we can identify $\cubichyp$ with $\Hyp\times\C$, where $\C$ is a copy of the fiber $\cubichyp_z$ over a point $z\in\Hyp$. We can define an $\SL(2,\R)$-action on $\Hyp\times \C$ by
$$\begin{pmatrix}
a & b \\ c & d
\end{pmatrix}\cdot(z,w):=\bigg(\frac{az+b}{cz+d}, (cz+d)^3w\bigg),\qquad \text{with} \ (z,w)\in\Hyp\times\C,\quad ad-bc=1 \ . $$
Moreover, the metric on the fiber is the one induced by the norm 
$$|w|_z^2=\Imm(z)^3|w|^2 \qquad \text{for} \ z\in\Hyp, w\in\cubichyp_z \ .$$ 
Given $J\in\almost$, let us define the space of $J$-complex symmetric tri-linear forms by \begin{align*}S_3(\R^2, J):&=\{\gamma:\R^2\otimes\R^2\otimes\R^2\longrightarrow\C \ | \ \gamma \ \text{ is symmetric and} \ (J,i)-\text{tri-linear}\} \\ &\cong\{\tau:\R^2\to\C \ | \ \text{for all} \ \alpha,\beta\in\R \ \text{and} \ v\in\R^2 \ \text{it holds} \ \tau(\alpha v+\beta Jv)=(\alpha+i\beta)^3\tau(v)\} \ . 
\end{align*}
This space can be seen as the fiber of a complex line bundle $\mathcal{L}_3(\R^2)\to\almost$ endowed with a natural $\SL(2,\R)$-action given by $$P\cdot (J,\gamma):=(PJP^{-1}, (P^{-1})^*\gamma),\qquad \ \text{for} \ P\in\SL(2,\R) \ .$$
It is not difficult to see that the line bundle $\mathcal{L}_3(\R^2)$ can be identified with $\pick$. In particular, each fiber $S_3(\R^2, J)$ is endowed with a scalar product from the one on $\pick_J$ defined in (\ref{def:scalarproductpick}).
\begin{lemma}[{\cite[Lemma 5.2.1]{trautwein2018infinite}}]\label{lem:Trautweinfibremap}
Define the map $\varphi:\cubichyp\to\Hom(\R^2\otimes\R^2\otimes\R^2, \C)$ by
\begin{align*}
    \varphi(z,w) : \ &\R^2\longrightarrow\C \\ & v\mapsto \bar w(v_1-\bar zv_2)^3
\end{align*}and let $j:\Hyp\to\almost$ be the map defined by (\ref{kahlerisometryalmost}). Then, the following holds: \begin{itemize}
    \item[(1)]$\varphi(z,w)\in S_3(\R^2,j(z)), \ \text{for all} \  (z,w)\in\cubichyp$.
    \item[(2)]The fibre map $\varphi(z,\cdot) : \cubichyp_z\cong\C\to S_3(\R^2,j(z))$ is a complex anti-linear isometry for every $z\in\Hyp$.
    \item[(3)] The bundle map $(j,\varphi): \cubichyp\to\mathcal{L}_3(\R^2)$ is a $\SL(2,\R)$-equivariant bijection.
\end{itemize}
\end{lemma}
At this point it easy to compute in coordinates the Pick tensor $C\in\pick_J$, the Pick form $A\in\End(\R^2)\otimes T^*\R^2$ and their respective variations: $\dot C$ and $\dot A=g_J^{-1}\dot C$, by using this last two lemmas and the isomorphism $\cubic\cong\pick$. Let $z=x+iy$ and $w=u+iv$ be the complex coordinates on $\Hyp$ and $\C$ respectively, then the bundle map $(j,\varphi)$ in Lemma \ref{lem:Trautweinfibremap} is given by $$\Hyp\times\C\ni(z,w)\longmapsto\big(j(z),C_{(z,w)}\big)\in\pick$$where $C_{(z,w)}=\Ree(q_{(z,w)})$ with $q_{(z,w)}=\widebar w(\dx_0-\widebar z\dy_0)^3$ (see Corollary \ref{cor:bijectioncubicpick}). Hence, the Pick form $A_{(z,w)}$ will be recovered by (\ref{picktensorandpickform}). Since $\SL(2,\R)$ acts transitively on $\Hyp$, it is enough to compute the tensors at the point $(i,w)\equiv(0,1,u,v)$ for a generic $w\in \C$. The components of the Pick tensor $C_{(z,w)}$ are given by
\begin{align*}
    &C_{111}(z,w)=u,\qquad C_{112}(z,w)=-xu+yv,\qquad C_{122}(z,w)=ux^2-uy^2-2xyv, \\ &C_{222}(z,w)=-ux^3-vy^3+3(uy^2x+x^2yv) \ .
\end{align*}
The remaining components are determined by the four above since $C$ is totally-symmetric. Its variation at $\dot C_{(i,w)}$ at $(i,w)$ is
\begin{align*}
    &\dot C_{111}(i,w)=\dot u,\qquad \dot C_{112}(i,w)=-u\dot x+\dot v+v\dot y,\qquad \dot C_{122}(i,w)=-\dot u-2(u\dot y+v\dot x), \\ &\dot C_{222}(i,w)=-\dot v+3(u\dot x-v\dot y) \ . 
\end{align*}
The Pick form computed in $(i,w)$ is then \begin{align}\label{pickformcoordinate}
    A_{(i,w)}=\begin{pmatrix}
    u & v \\ v & -u
    \end{pmatrix}\dx_{0}+\begin{pmatrix}
    v & -u \\ -u & -v
    \end{pmatrix}\dy_{0} \ .
\end{align}
Its variation $\dot A$ will be given in terms of its trace-free and trace part at the point $(i,w)$
\begin{equation*}\label{dotpickformcoordinates}
     (\dot A_0)_{(i,w)}=\begin{pmatrix}
    \dot u+u\dot y+v\dot x & -u\dot x+\dot v+v\dot y \\ -u\dot x+\dot v+v\dot y & -\dot u-u\dot y-v\dot x
    \end{pmatrix}\dx_{0}+\begin{pmatrix}
    \dot v+2(v\dot y-u\dot x) & -\dot u-2(u\dot y+v\dot x) \\ -\dot u-2(u\dot y+v\dot x) & -\dot v+2(u\dot x-v\dot y)
    \end{pmatrix}\dy_{0}
\end{equation*}\begin{equation*}
    (\dot A_{\tr})_{(i,w)}=\frac 1{2}\begin{pmatrix}
    -u\dot y-v\dot x & 0 \\ 0 & -u\dot y-v\dot x
    \end{pmatrix}\dx_{0}+\frac 1{2}\begin{pmatrix}
    u\dot x-v\dot y & 0 \\ 0 & u\dot x-v\dot y
    \end{pmatrix}\dy_{0} \ . 
\end{equation*}
Thanks to this expression in coordinates and together with the action of $\SL(2,\R)$ on $\Hyp\times\C$, we are now able to write the metric $\g$ and the symplectic form $\ome$ at the point $(z,w)$. Let $\{\frac{\partial}{\partial x},\frac \partial{\partial y},\frac\partial{\partial u},\frac\partial{\partial v}\}$ be a real basis of the tangent space of $\Hyp\times\C$ with its dual basis $\{\dx,\dy,\du,\devu\}$, then the expressions (\ref{pseudoriemannianmetric}) and (\ref{symplecticform}) become respectively
\begin{equation*}
    \g_{(z,w)}=\begin{pmatrix}
    \frac 1{y^2}\big(1-f+3(u^2+v^2)y^3f'\big) & 0 & 2f'vy^2 & -2f'uy^2 \\ 0 & \frac 1{y^2}\big(1-f+3(u^2+v^2)y^3f'\big) & 2f'uy^2 & 2f'vy^2 \\ 2f'vy^2 & 2f'uy^2 & \frac 4{3}f'y^3 & 0 \\ -2f'uy^2 & 2f'vy^2 & 0 &\frac 4{3}f'y^3 
    \end{pmatrix}
\end{equation*}
\begin{align*}
    \ome_{(z,w)}=&\bigg(-1+f-3f'y^3(u^2+v^2)\bigg)\frac{\dx\wedge\dy}{y^2}-\frac 4{3}f'y^3\du\wedge\devu \\&-2y^2f'\bigg(u(\dx\wedge\du+\dy\wedge\devu)+v(\du\wedge\dy-\devu\wedge\dx)\bigg)
\end{align*}
where the functions $f,f'$ are evaluated in: $$\frac 1{4}\vl\vl A_{(z,w)}\vl\vl^2_{j(z)}=\vl\vl q_{(z,w)}\vl\vl^2_{j(z)}=y^3(u^2+v^2) \ .$$
The matrix associated with the complex structure $\i_{(z,w)}: T_{(z,w)}\big(\Hyp\times\C\big)\to T_{(z,w)}\big(\Hyp\times\C\big)$ in the basis $\{\frac{\partial}{\partial x},\frac \partial{\partial y},\frac\partial{\partial u},\frac\partial{\partial v}\}$ is \begin{equation*}
    \i_{(i,w)}=\begin{pmatrix}
    J_0 & 0_{2\times 2} \\ 0_{2\times 2} & -J_0
    \end{pmatrix}=\begin{pmatrix}
    0 & -1 & 0 & 0 \\ 1 & 0 & 0 & 0 \\ 0 & 0 & 0 & 1 \\ 0 & 0 & -1 & 0
    \end{pmatrix} \ . 
\end{equation*}
We will explain how to obtain the expressions above for $\g$ and $\ome$ later in the section. We first show that these formulas define a non-degenerate pseudo-Riemannian metric and a closed $2$-form on $\Hyp \times \C$, thus concluding the proof of Theorem \ref{thmA}. 
\begin{lemma}\label{lem:nondegenerateclosed}
The tensor $\g_{(z,w)}$ is non-degenerate and the form $\ome_{(z,w)}$ is closed, for each $(z,w)\in\Hyp\times\C$.
\end{lemma}
\begin{proof}
The tensor $\g$ can be written as: $$\g_{(z,w)}=\begin{pmatrix}
A & B \\ C & D
\end{pmatrix}$$where $A,B,C,D$ are $2\times 2$ matrices with $$A=\frac 1{y^2}\big(1-f+3y^3(u^2+v^2)f'\big)\dsone_{2\times 2}, \qquad\qquad D=\frac 4{3}y^3f'\dsone_{2\times 2} \ .$$
Hence, $B$ and $C$ both commute with $A$ and $D$. In this case there is an easy formula for the determinant of the $4\times 4$ matrix, namely $\det(\g_{(z,w)})=\det(AD-BC)$, where $$AD=\frac 4{3}y\big(f'-ff'+3y^3(f')^2(u^2+v^2)\big)\dsone_{2\times 2} \qquad BC=4y^4(f')^2(u^2+v^2)\dsone_{2\times 2}$$which gives $$\det(g_{(z,w)})=\frac {16}{9}y^2(f')^2(1-f)^2 \ . $$
The right hand side of the last equation is always non-zero thanks to the property of the function $f$ (see Definition \ref{def:pseudoriemannian}), hence $\g_{(z,w)}$ is non-degenerate at each point $(z,w)\in\Hyp\times\C$.\newline It only remains to prove that $(\mathrm{d}\ome)_{(z,w)}=0$ for each $(z,w)\in\Hyp\times\C$. By using directly the expression in coordinate, we get: \begin{itemize}
    \item Coefficient $\dy\wedge\du\wedge\devu$:\begin{align*}
         &-4y^2f'\dy\wedge\du\wedge\devu-4y^5f''(u^2+v^2)\dy\wedge\du\wedge\devu-4y^5f''u^2\du\wedge\dy\wedge\devu \\& -2y^2f'\du\wedge\dy\wedge\devu-4y^5f''v^2\devu\wedge\du\wedge\dy-2y^2f'\devu\wedge\du\wedge\dy=0
    \end{align*} \item Coefficient $\dx\wedge\du\wedge\devu$:\begin{equation*}
        -4y^5f''uv\devu\wedge\dx\wedge\du+4y^5f''uv\du\wedge\devu\wedge\dx=0
    \end{equation*}\item Coefficient $\dx\wedge\dy\wedge\devu$:\begin{align*}
        &2yf'v\devu\wedge\dx\wedge\dy-6y^4f''u^2v\devu\wedge\dx\wedge\dy-6yf'v\devu\wedge\dx\wedge\dy \\ &-6y^4f''v^3\devu\wedge\dx\wedge\dy+4yf'v\dy\wedge\devu\wedge\dx+6y^4f''v(u^2+v^2)\dy\wedge\devu\wedge\dx=0
    \end{align*}
    \item Coefficient $\dx\wedge\dy\wedge\du$: \begin{align*}
        &2yf'u\du\wedge\dx\wedge\dy-6yf'u\du\wedge\dx\wedge\dy-6y^4f''u^3\du\wedge\dx\wedge\dy \\ &-6y^4f''uv^2\du\wedge\dx\wedge\dy-4yf'u\dy\wedge\dx\wedge\du-6y^4f''(u^2+v^2)\dy\wedge\dx\wedge\du=0
    \end{align*}
\end{itemize}
\end{proof}
Thanks to the expression in coordinates, it is easy to see that $\g$ is indeed a pseudo-Riemannian metric on $\Hyp\times\C$ (hence on $\deft$) since it is negative-definite when restricted to $\{0\}\times\C$ and it coincides with $g_{\Hyp}$ on $\Hyp\times\{0\}$. \\

In the following we will give an idea on how to compute $\g_{(i,w)}$ by using Definition \ref{def:pseudoriemannian} and the expression of the tensors in coordinates. Finally, by using the $\SL(2,\R)$-invariance we briefly sketch how to compute the tensor $\g$ at an arbitrary point of $\Hyp\times\C$.\newline In order to simplify the computation we will give the expression of the associated quadratic form. The first part of the quadratic form in the tensor formalism is $(1-f)\frac 1{2}\tr(\dot J^2)$, hence at the point $(i,w)$ we have: \begin{equation*}
    J=J_0=\begin{pmatrix}0 & -1 \\ 1 & 0
\end{pmatrix} \qquad \text{and} \qquad \dot J=\mathrm{d}_ij(\dot x,\dot y)=\begin{pmatrix}
\dot x & -\dot y \\ -\dot y & -\dot x
\end{pmatrix} \ . \end{equation*}
Thus, $(1-f)\frac 1{2}\tr(\dot J^2)=(1-f)\big(\dot x^2+\dot y^2\big)$. Moreover, using the expression in coordinates of $(\dot A_0)_{(i,w)}$ and $(\dot A_{\tr})_{(i,w)}$ we get: \begin{align*}
    &(\ast\dot A_0)_{(i,w)}=\begin{pmatrix}
    \dot u+u\dot y+v\dot x & -u\dot x+\dot v+v\dot y \\ -u\dot x+\dot v+v\dot y & -\dot u-u\dot y-v\dot x
    \end{pmatrix}\dy_{0}-\begin{pmatrix}
    \dot v+2(v\dot y-u\dot x) & -\dot u-2(u\dot y+v\dot x) \\ -\dot u-2(u\dot y+v\dot x) & -\dot v+2(u\dot x-v\dot y)
    \end{pmatrix}\dx_{0}, \\ &(\ast\dot A_{\tr})_{(i,w)}=\frac 1{2}\begin{pmatrix}
    -u\dot y-v\dot x & 0 \\ 0 & -u\dot y-v\dot x
    \end{pmatrix}\dy_{0}-\frac 1{2}\begin{pmatrix}
    u\dot x-v\dot y & 0 \\ 0 & u\dot x-v\dot y
    \end{pmatrix}\dx_{0} \ .
\end{align*}Hence, \begin{align*}
    &\frac{1}{3}\tr(\dot A_0\wedge\ast\dot A_0)=\frac 4{3}(\dot u^2+\dot v^2)+\frac{10}{3}(u^2+v^2)(\dot x^2+\dot y^2)+4\big(u(\dot u\dot y-\dot x\dot v)+v(\dot y\dot v+\dot u\dot x)\big) \\ &\frac 1{6}\tr(\dot A_{\tr}\wedge\ast\dot A_{\tr})=\frac 1{3}(u^2+v^2)(\dot x^2+\dot y^2) \ .
\end{align*}The final expression for the quadratic form associated with $\g$ and computed at $(i,w)$ is thus
\begin{equation}\label{quadraticform}\big(1-f+3f'(u^2+v^2)\big)\big(\dot x^2+\dot y^2\big)+\frac 4{3}f'(\dot u^2+\dot v^2)+4f'\big(u(\dot u\dot y-\dot x\dot v)+v(\dot y\dot v+\dot u\dot x)\big)\end{equation}It is now clear how to recover the above expression of $\g_{(i,w)}$ from (\ref{quadraticform}). In order to give the precise expression of $\g$ at an arbitrary point $(z,\widetilde w)\in\Hyp\times\C$ we need to use the $\SL(2,\R)$-invariance of $\g$ and the fact that the $\SL(2,\R)$-action on $\Hyp$ is transitive. In fact, we can find a $P\in\SL(2,\R)$ such that $P\cdot z=i$ for $z\in\Hyp$, where $P\cdot z$ is the action via M\"obius transformations. This matrix $P$ is explicitly given by$$
P=\begin{pmatrix}
\frac 1{\sqrt y} & -\frac{x}{\sqrt y} \\ 0 & \sqrt y
\end{pmatrix} \ .$$ In particular, the point $\widetilde w=\widetilde u+i\widetilde v\in\C$ is determined by $P\cdot (z,\widetilde w)=(i,w)$. In fact,  \begin{equation}\label{eq:sum_of_squares}
    u^2+v^2=y^3(\widetilde u^2+\widetilde v^2) \ .
\end{equation}
By using the $\SL(2,\R)$-invariance we get $$\g_{(z,\widetilde w)}(\cdot,\cdot)=(P^*\g)_{(z,\widetilde w)}(\cdot,\cdot)=\g_{(i,w)}(\mathrm d_{(z,\widetilde w)}P\cdot,\mathrm d_{(z,\widetilde w)}P\cdot), $$
where the differential of $P$ at $(z,\widetilde w)$ is given by
\begin{align*}&\mathrm d_{(z,\widetilde w)}P\bigg(\frac\partial{\partial x}\bigg)=\frac 1{y}\frac\partial{\partial x}\quad &\mathrm d_{(z,\widetilde w)}P\bigg(\frac\partial{\partial y}\bigg)&=\frac 1{y}\frac\partial{\partial y} \\ &\mathrm d_{(z,\widetilde w)}P\bigg(\frac\partial{\partial u}\bigg)=y^{\frac 3{2}}\frac\partial{\partial u} \quad &\mathrm d_{(z,\widetilde w)}P\bigg(\frac\partial{\partial v}\bigg)&=y^{\frac 3{2}}\frac\partial{\partial v} \ . \end{align*}
Now we have all the tools to compute $\g$ at a point $(z,\widetilde w)$. For instance, \begin{align*}
    \g_{(z,\widetilde w)}\bigg(\frac\partial{\partial x},\frac\partial{\partial x}\bigg)&=\g_{(i,w)}\bigg(\mathrm d_{(z,\widetilde w)}P\bigg(\frac\partial{\partial x}\bigg),\mathrm d_{(z,\widetilde w)}P\bigg(\frac\partial{\partial x}\bigg)\bigg) \\&=\frac 1{y^2}\g_{(i,w)}\bigg(\frac\partial{\partial x},\frac\partial{\partial x}\bigg) \\ &=\frac 1{y^2}\bigg(1-f+3f'(u^2+v^2)\bigg) \\
    &=\frac 1{y^2}\bigg(1-f+3y^{3}f'(\widetilde u^2+\widetilde v^2)\bigg)\ \tag{Equation \ref{eq:sum_of_squares}} 
\end{align*}With a similar computation one can recover all the entries of the tensor $\g$ at every $(z,w)\in\Hyp\times\C$.

\subsection{The circle action and the \texorpdfstring{$\SL(2,\R)$}{SL(2,R)}-moment map}
In this section we study the behavior of the circle and $\SL(2,\R)$-action on $\deft$ with respect to the pseudo-K\"ahler structure $(\g,\i,\ome)$. Before analyzing the two actions in detail let us recall some basic definitions regarding symplectic actions and moment maps.\begin{definition}
Let $(M,\omega)$ be a symplectic manifold and let $G$ be a Lie group acting on $M$. Let $\psi_g:M\to M$ be the map $\psi_g(p):=g\cdot p$, then we say the group $G$ acts by symplectomorphisms on $(M,\omega)$ if $\psi_g^*\omega=\omega$ for all $g\in G$.
\end{definition}  
\begin{definition}\label{def:momentmap}
Let $G$ be a Lie group, with Lie algebra $\mathfrak g$, acting on a symplectic manifold $(M,\omega)$ by symplectomorphisms. We say the action is \emph{Hamiltonian} if there exists a smooth function $\mu:M\to\mathfrak g^*$ satisfying the following properties: \begin{itemize}
    \item[(i)] The function $\mu$ is equivariant with respect to the $G$-action on $M$ and the co-adjoint action on $\mathfrak g^*$, namely \begin{equation}
        \mu_{g\cdot p}=\Ad^*(g)(\mu_p):=\mu_p\circ \Ad(g^{-1})\in\mathfrak g^* \ .
    \end{equation}
    \item[(ii)]Given $\xi\in\mathfrak g$, let $X_\xi$ be the vector field on $M$ generating the action of the $1$-parameter subgroup generated by $\xi$, i.e. $X_\xi=\frac{\mathrm d}{\mathrm dt}\text{exp}(t\xi)\cdot p |_{t=0}$. Then, for every $\xi\in\mathfrak g$ we have\begin{equation}
        \mathrm d\mu^{\xi}=\iota_{X_\xi}\omega=\omega(X_\xi,\cdot)
    \end{equation}where $\mu^\xi:M\to\R$ is the function $\mu^\xi(p):=\mu_p(\xi)$.
\end{itemize}A map $\mu$ satisfying the two properties above is called a \emph{moment map} for the Hamiltonian action.
\end{definition}
Let us start with the study of the $S^1$-action on $\deft$. First of all, we need to understand how the circle action $q\mapsto e^{-i\theta}q$ on $Q^3_0(\Teich^c(T^2))$ changes under the bijection with $\pick$ (see Corollary \ref{cor:bijectioncubicpick}). In other words, if $C$ is the Pick tensor associated with the $J$-holomorphic cubic differential $q$, namely $C=\Ree(q)$, then we need to find the expression of the new Pick form $\widetilde A$ associated with $\widetilde C=\Ree(e^{-i\theta}q)$. We have $q=C(\cdot,\cdot,\cdot)-iC(\cdot,J\cdot,\cdot)$ since $(J,C)\in\pick$ by definition. In particular, the expression
$$e^{-i\theta}q=\cos\theta C(\cdot,\cdot,\cdot)+\sin\theta C(\cdot,J\cdot,\cdot)+i\bigg(\cos\theta C(\cdot,J\cdot,\cdot)-\sin\theta C(\cdot,\cdot,\cdot)\bigg)$$
implies that $\widetilde C(\cdot,\cdot,\cdot)=\cos\theta C(\cdot,\cdot,\cdot)+\sin\theta C(\cdot,J\cdot,\cdot)$ and the new Pick form is $$\widetilde A(\cdot)=g_J^{-1}\widetilde C=\cos\theta A(\cdot)-\sin\theta A(\cdot)J \ . $$
The last equation gives an induced action on $\pick$ by setting
\begin{align*}
    \Psi_\theta \ :& \ \pick\longrightarrow\pick \\ &(J,A)\mapsto (J,\cos\theta A(\cdot)-\sin\theta A(\cdot)J) \ . \ 
\end{align*}
It is clear from the definition that $\Psi_\theta$ preserves the $0$-section in $\pick$ (seen as a vector bundle over $\almost$), hence it induces an $S^1$-action on $\deft$ which will still be denoted by $\Psi_\theta$ by abuse of notation. Before stating and proving the main result, we need a technical lemma regarding the derivative of the norm of the Pick form.
\begin{lemma}\label{lem:derivativepickform}
Let $(J,A)\in\pick$, then \begin{equation}
    \big(\vl\vl A\vl\vl_J^2\big)'=2\langle A,\dot A_0\rangle \ . 
\end{equation}
\end{lemma}\begin{proof}
 During the proof of this lemma we use the notation of the previous section, namely $A=A_1e_1^*+A_2e_2^*$ and $\dot A=\dot A_1e_1^*+\dot A_2e_2^*$, with $\{e_1,e_2\}$ a $g_J$-orthonormal basis of $\R^2$ and $\{e_1^*,e_2^*\}$ its dual basis. Recall that the relation between the Pick form $A$ and the Pick tensor $C$ is $A=g_J^{-1}C$, hence \begin{align*}
     A'&=(g_J^{-1}C)' \\&=-g_j^{-1}\dot g_Jg_J^{-1}C+g_J^{-1}\dot C \\&=J\dot JA+\dot A \tag{$\dot g_J(\cdot,\cdot)=-g_J(\cdot,J\dot J\cdot)$}
 \end{align*}In particular, $(A_i)'=J\dot JA_i+\dot A_i$, for each $i=1,2$. Thus, \begin{align*}
     \big(\vl\vl A\vl\vl_J^2\big)'&=\int_{T^2}\bigg(\tr(A'\wedge\ast A)+\tr(A\wedge\ast A')\bigg) \\ &=\int_{T^2}\bigg(2\tr(A_1\dot A_1+A_2\dot A_2)+\tr(J\dot J((A_1)^2+(A_2)^2)+A_1J\dot JA_1+A_2J\dot JA_2)\bigg)e_1^*\wedge e_2^* \\ &=2\int_{T^2}\tr(A_1\dot A_1+A_2\dot A_2)e_1^*\wedge e_2^* \ ,
 \end{align*}
 where in the second line we used the fact that, since both $J\dot J A_iA_i$ and $A_iJ\dot JA_i$ anticommute with $J$ for each $i=1,2$, the terms $\tr(J\dot JA_iA_i)=\tr(A_iJ\dot JA_i)$ vanish for each $i=1,2$. Thus, we get $$\big(\vl\vl A\vl\vl_J^2\big)'=2\langle A,\dot A\rangle \ . $$
 Finally, by writing $\dot A=\dot A_0+\dot A_{\tr}$ where $$\dot A_{\tr}=\frac 1{2}\tr(\dot A_1)\dsone e_1^*+\frac 1{2}\tr(\dot A_2)\dsone e_2^*$$ we obtain $$\langle A,\dot A_{\tr}\rangle=\frac 1{2}\int_{T^2}\tr(\tr(\dot A_1)A_1+\tr(\dot A_2)A_2)e_1^*\wedge e_2^*$$
and this last term is equal to zero since the $A_i$'s are trace-less endomorphisms. 
\end{proof}
\begin{manualtheorem}B
The $S^1$-action on $\deft$ is Hamiltonian with respect to $\ome$ and it satisfies \begin{equation*}
     \Psi_\theta^{*}\ome=\ome \qquad \text{and} \qquad \Psi_\theta^{*}\g=\g \ . 
\end{equation*} The Hamiltonian function is given by $H\big(\normpick\big)=\frac 2{3}f\big(\normpick\big). $
\end{manualtheorem}
\begin{proof}
 The infinitesimal generator of the action is $$X_{(J,A)}=\frac{\mathrm{d}}{\mathrm{d}\theta}_{\big|_{\theta=0}}\Psi_\theta(J,A)=(0,-AJ).$$
Hence, \begin{align*}
(\iota_X\ome)_{(J,A)}(\dot J,\dot A)&=\ome_{(J,A)}\big((\dot J,\dot A), (0, AJ)\big) \\ &=\g_{(J,A)}\big((\dot J,\dot A),\i(0,AJ)\big) \\&=\g_{(J,A)}\big((\dot J,\dot A),(0,-AJ^2)\big) \\ &=\frac{f'}{3}\langle\dot A_0,A\rangle_J \  \tag{$A \ \text{is} \ g_J-\text{traceless}$}
\end{align*}
Now we compute the differential of $H\big(\vl\vl A\vl\vl_J^2\big)=\frac 2{3}f\big(\normpick\big)$. This is given by
\begin{align*}
    \mathrm{d}_{(J,A)}H(\dot J,\dot A)&=\frac{f'}{6}\big(\vl\vl A\vl\vl_J^2\big)' \\ &=\frac{f'}{3}\langle A,\dot A_0\rangle_J \tag{\text{Lemma} \ \ref{lem:derivativepickform}}
\end{align*}
Thus, the $S^1$-action is Hamiltonian. The symplectic form $\ome$ is preserved by this action since its Lie derivative along the infinitesimal generator $X$ is zero: $$\Dlie_X\ome=\mathrm{d}(\iota_X\ome)+\iota_X\mathrm{d}\ome=\mathrm{d}^2H=0$$by Cartan's magic formula.\newline It only remains to prove that $\Psi_\theta$ is an isometry for $\g$. First of all we compute the differential of the action: $$\mathrm{d}_{(J,A)}\Psi_\theta(\dot J,\dot A)=\big(\dot J,\cos\theta\dot A(\cdot)-\sin\theta(\dot A(\cdot)J+A(\cdot)\dot J)\big) \ .$$
Then, we notice that the circle action preserves the norm of the Pick form, namely $\vl\vl\cos\theta A-\sin\theta AJ\vl\vl^2_J=\vl\vl A\vl\vl_J^2$. In fact, \begin{equation}\label{proofthmB}\vl\vl\cos\theta A-\sin\theta AJ\vl\vl^2_J=\cos^2\theta\vl\vl A\vl\vl_J^2+\sin^2\theta\underbrace{\vl\vl AJ\vl\vl_J^2}_{(a)}-2\cos\theta\sin\theta\underbrace{\langle A,AJ\rangle_J}_{(b)} \ . \end{equation}
The term $(a)$ is \begin{align*}
    \vl\vl AJ\vl\vl_J^2&=\int_{T^2}\tr(A_1JA_1J+A_2JA_2J) \ e_1^*\wedge e_2^*\\ &=\int_{T^2}\tr(A_1A_1+A_2A_2) \ e_1^*\wedge e_2^*\tag{$A_i\in T_J\almost \ \text{and} \ J^2=-\dsone$} \\ &=\vl\vl A\vl\vl_J^2 \ .
\end{align*}The term $(b)$ is \begin{align*}
    \langle A,AJ\rangle_J=\int_{T^2}\tr(A_1A_1J+A_2A_2J) \ e_1^*\wedge e_2^*
\end{align*}but $\tr(A_iJA_i)=\tr(JA_iA_i)=\tr(A_iA_iJ)=-\tr(A_iJA_i)$ for $i=1,2$, hence the term $(b)$ is zero. In the first two equalities we used the trace symmetry and in the third one the fact that $A_i\in T_J\almost$. The circle action $\Psi_\theta$ preserves the pseudo-Riemannian metric $\g$ if and only if the following holds
$$\g_{(J,A)}\big((\dot J,\dot A); (\dot J',\dot A')\big)=\g_{\Psi_\theta(J,A)}\big(\mathrm{d}_{(J,A)}\Psi_\theta(\dot J,\dot A); \mathrm{d}_{(J,A)}\Psi_\theta(\dot J',\dot A')\big) \ . $$ Let us define $\psi_\theta$ to be the second component of the differential of the circle action, namely $$\psi_\theta(\dot J,\dot A):=\cos\theta\dot A-\sin\theta(\dot AJ+A\dot J).$$
Then, in order to conclude the proof, we need to show
\begin{itemize}
    \item[(1)] $\langle\psi_\theta(\dot J,\dot A)_0, \psi_\theta(\dot J', \dot A')_0\rangle=\langle\dot A_0,\dot A_0'\rangle ;$
    \item[(2)]$\langle\psi_\theta(\dot J,\dot A)_{\tr}, \psi_\theta(\dot J', \dot A')_{\tr}\rangle=\langle\dot A_{\tr},\dot A_{\tr}'\rangle .$
\end{itemize}The left hand side term of $(1)$ can be written as
\begin{align*}\cos^2\theta\langle\dot A_0,\dot A_0'\rangle&+\sin^2\theta\langle (\dot AJ+A\dot J)_0, (\dot A'J'+A'\dot J')_0\rangle \\ &+\cos\theta\sin\theta\big(\langle\dot A_0,-(\dot A'J+A\dot J')_0\rangle+\langle-(\dot AJ+A\dot J)_0,\dot A_0'\rangle\big) \ .\end{align*}The coefficient of $\sin^2\theta$ has already been calculated (see proof of Theorem \ref{thm:pseudokahlerpick}) and it is equal to $\langle\dot A_0,\dot A_0'\rangle$. The coefficient of $\cos\theta\sin\theta$ is equal to
\begin{align*}
    \int_{T^2}\tr(-(\dot A_1)_0(\dot A_1')_0J-(\dot A_2)_0(\dot A_2')_0J-(\dot A_1)_0J(\dot A_1')_0-(\dot A_2)_0J(\dot A_2')_0) \ e_1^*\wedge e_2^*
\end{align*}and it vanishes since $(\dot A_i)_0,(\dot A_i')_0\in T_J\almost$ for each $i=1,2$.\newline The left hand side term of $(2)$ can be written as \begin{align*}\cos^2\theta\langle\dot A_{\tr},\dot A_{\tr}'\rangle&+\sin^2\theta\langle (\dot AJ+A\dot J)_{\tr}, (\dot A'J'+A'\dot J')_{\tr}\rangle \\ &+\cos\theta\sin\theta\big(\langle\dot A_{\tr},-(\dot A'J+A\dot J')_{\tr}\rangle+\langle-(\dot AJ+A\dot J)_{\tr},\dot A_{\tr}'\rangle\big) \ . \end{align*}
The coefficient of $\sin^2\theta$ has already been calculated (see proof of Theorem \ref{thm:pseudokahlerpick}) and it is equal to $\langle \dot A_{\tr},\dot A_{\tr}'\rangle$. By using Lemma \ref{lem:decompositioncomplexstructure}, the coefficient of $\cos\theta\sin\theta$ can be written as \begin{align*}
    \frac 1{2}\int_{T^2}\tr(-(\dot A_1)_{\tr}\tr(\dot A_2')\dsone+(\dot A_2)_{\tr}\tr(\dot A_1')\dsone-(\dot A_1')_{\tr}\tr(\dot A_2)\dsone+(\dot A_2')_{\tr}\tr(\dot A_1)\dsone)e_1^*\wedge e_2^* \ .
\end{align*}Since $(\dot A_i)_{\tr}=\frac 1{2}\tr(\dot A_i)\dsone$ and $(\dot A_i')_{\tr}=\frac 1{2}\tr(\dot A_i')\dsone$ for each $i=1,2$, the first term of the above expression cancels out with the last one and the same happens for the second and third one. Finally, the term of $\cos\theta\sin\theta$ vanishes and we obtain the thesis. 
\end{proof}
Now we will study the $\SL(2,\R)$-action and its moment map. Recall that if $P\in\SL(2,\R)$ and $(J,A)\in\pick$, then $$P\cdot (J,A)=(PJP^{-1},PA(P^{-1}\cdot)P^{-1}) \ . $$ In particular, this action preserves the $0$-section in $\pick$ (seen as a vector bundle over $\almost$), hence it induces an $\SL(2,\R)$-action on $\deft$, which will be denoted by $\Phi_P:\deft\to\deft$. Thanks to Lemma \ref{lem:invarianza}, it is clear that $\Phi_P^*\ome=\ome$, i.e. $\SL(2,\R)$ acts by symplectomorphisms on $\deft$. Thus, it makes sense to ask if this action is Hamiltonian and, if this is the case, to give the expression of the moment map. The Lie algebra of $\SL(2,\R)$ is given by $\Lsl(2,\R)=\{X\in\End(\R^2) \ | \ \tr(X)=0\}$ with Lie bracket $[X,Y]=XY-YX$. In particular any $X\in\Lsl(2,\R)$ can be decomposed as $X=X^a+X^s$, where $X^s$ is a trace-less $g_J$-symmetric matrix and $X^a$ is a trace-less $g_J$-skew-symmetric matrix. In particular, $X^s\in T_J(\almost)$ and $X^a=-\frac 1{2}\tr(JX)J$, since it commutes with $J$.
\begin{manualtheorem}C
The $\SL(2,\R)$-action on $\deft$ is Hamiltonian with respect to $\ome$ with moment map $\mu:\deft\to\Lsl(2,\R)^*$ given by
\begin{equation}
    \mu_{(J,A)}(X)=\bigg(1-f\big(\normpick\big)\bigg)\tr(JX)
\end{equation}for all $X\in\Lsl(2,\R)$.
\end{manualtheorem}
\begin{proof}
     $\bullet$ \underline{\emph{The infinitesimal generator of the action}}. Let $X\in\Lsl(2,\R)$ and let $$V_X(J,A)=\frac{\mathrm d}{\mathrm dt}(e^{tX}Je^{-tX}, (e^{-tX})^{*}C)|_{t=0}$$ be its infinitesimal generator. The first component is equal to $XJ-JX=[X,J]$. For the second component define $P_t:=e^{tX}$, then we need to compute the Pick form corresponding to \begin{align*}
         \frac{\mathrm d}{\mathrm dt}C((P_t)^{-1}\cdot,(P_t)^{-1}\cdot,(P_t)^{-1}\cdot)|_{t=0}=-C(X\cdot,\cdot,\cdot)-C(\cdot,X\cdot,\cdot)-C(\cdot,\cdot,X\cdot) \ . 
     \end{align*}
     If $\widetilde C(\cdot,\cdot,\cdot)$ is defined as the right hand side term of the equation above, then the new Pick form $\widetilde A$ satisfies \begin{align*}
         g_J(\widetilde A(Y)Z,W)&=\widetilde C(Y,Z,W) \\ &=-C(X\cdot Y,Z,W)-C(Y,X\cdot Z,W)-C(Y,Z,X\cdot W) \\ &=-g_J(A(X\cdot Y)Z,W)-g_J(A(Y)X\cdot Z,W)-g_J(A(Y)Z,X\cdot W) \\ &=-g_J(A(X\cdot Y)Z+A(Y)X\cdot Z+X^*\cdot A(Y)Z, W)
     \end{align*}for all $Y,Z,W\in\Gamma(T\R^2)$, where $X^*$ denotes the adjoint of $X$ with respect to $g_J$. Hence, we have 
     $$\widetilde A(\cdot)=-A(X\cdot)-AX-X^*A \ .$$ By using the decomposition $X=X^s+X^a$ in its symmetric and skew-symmetric part, we can write the second component of $V_X(J,A)$ as: \begin{equation}\label{decompositiontracepart}
     \underbrace{-A(X^a\cdot)-A(X^s\cdot)+[X^a,A]}_{\text{trace-less part}}-\underbrace{(AX^s+X^sA)}_{\text{trace part}} \ . \end{equation}
    $\bullet$ \underline{$\mu$ \emph{is equivariant}}. Let $P\in\SL(2,\R)$ and $X\in\Lsl(2,\R)$, then \begin{align*}
         \mu_{P\cdot (J,A)}(X) &=\bigg(1-f\bigg(\frac 1{4}\vl\vl P\cdot A\vl\vl_{P\cdot J}^2\bigg)\bigg)\tr(PJP^{-1}X) \\ &=\bigg(1-f\big(\normpick\big)\bigg)\tr(JP^{-1}XP) \\ &=\mu_{(J,A)}\circ\Ad(P^{-1})(X) \\ &=\Ad^*(P)(\mu_{(J,A)})(X)
     \end{align*}where in the second equality we used Lemma \ref{lem:invarianza} and the trace symmetry. \\
     $\bullet$ \underline{$\mu$ \emph{satisfies property (ii) in Definition} \ref{def:momentmap}}.
     Let $\mu^X:\deft\to\R$ be the map $$\mu^X(J,A)=\bigg(1-f\big(\normpick\big)\bigg)\tr(JX) \ , $$ then \begin{align*}
         \mathrm{d}_{(J,A)}\mu^X(\dot J,\dot A)&=-\frac 1{4}\bigg(\vl\vl A\vl\vl^2_J\bigg)'f'\big(\normpick\big)\tr(JX)+\bigg(1-f\big(\normpick\big)\bigg)\tr(\dot JX) \\ &=-\frac 1{2}\langle A,\dot A_0\rangle f'\big(\normpick\big)\tr(JX)+\bigg(1-f\big(\normpick\big)\bigg)\tr(\dot JX)
     \end{align*}where we used Lemma \ref{lem:derivativepickform} in the second equality. Now let $V_X$ be the infinitesimal generator of $X$, then \begin{equation}\label{iotaomega}\begin{aligned}
         \iota_{V_X}\ome(\dot J,\dot A)&=\g(V_X(J,A), \i(\dot J,\dot A)) \\&=\frac{f-1}{2}\tr([X,J]J\dot J)+\frac{f'}{3}\langle [X^a, A]-A(X\cdot), (-\dot AJ-A\dot J)_0\rangle \\ &-\frac {f'}{6}\langle AX^s+X^sA, (\dot AJ+A\dot J)_{\tr}\rangle
     \end{aligned}\end{equation}where we used the decomposition in (\ref{decompositiontracepart}). The first term of $\iota_{V_X}\ome(\dot J,\dot A)$ is
     \begin{align*}
         \frac{1-f}{2}\tr(\dot JX+JXJ\dot J)=(1-f)\tr(\dot J X)
     \end{align*}by trace symmetry and $\dot J J+J\dot J=0$. It only remains to show that the sum of the second and third term of $\iota_{V_X}\ome(\dot J,\dot A)$ is equal to $-\frac 1{2}f'\langle A,\dot A_0\rangle\tr(JX)$. The coefficient of $\frac{f'}{3}$ in (\ref{iotaomega}) can be written as \begin{equation*}
         \underbrace{\langle A(X^s\cdot), (\dot AJ+A\dot J)_0\rangle}_{(a)}+\underbrace{\langle [X^a, A]-A(X^a\cdot), (-\dot AJ-A\dot J)_0\rangle}_{(b)} \ .
     \end{equation*}Moreover, by using Lemma \ref{lem:decompositioncomplexstructure} the term with $[X^a,A]$ in $(b)$ becomes \begin{align*}
         \frac 1{2}\int_{T^2}\tr(JX)\tr(JA_1(\dot A_1)_0J+JA_2(\dot A_2)_0J-A_1J(\dot A_1)_0J-A_2J(\dot A_2)_0J) e_1^*\wedge e_2^* \ . 
     \end{align*}
     Using that $A_i,(\dot A_i)\in T_J\almost$ for each $i=1,2$, the above term reduces to $-\tr(JX)\langle A,\dot A_0\rangle$. Notice that $-A(X^a\cdot)=\frac 1{2}\tr(JX)A(\cdot)J$, since $C(J\cdot,\cdot,\cdot)=C(\cdot,J\cdot,\cdot)$. Hence, the term with $-A(X^a\cdot)$ in $(b)$ becomes
     \begin{align*}
         -\frac 1{2}\int_{T^2}\tr(JX)\tr(A_1J(\dot A_1)_0J+A_2J(\dot A_2)_0J)e_1^*\wedge e_2^*=-\frac 1{2}\tr(JX)\langle A,\dot A_0\rangle \ . 
     \end{align*}
     Finally, the term $(b)$ multiplied by $\frac{f'}{3}$ is equal to
     $$-\frac 1{2}f'\tr(JX)\langle A,\dot A_0\rangle \ .$$ Hence, it only remains to show that \begin{equation}\label{lastequation}
         \frac{f'}{3}\underbrace{\langle A(X^s\cdot), (\dot AJ+A\dot J)_0\rangle}_{(a)}-\frac {f'}{6}\underbrace{\langle AX^s+X^sA, (\dot AJ+A\dot J)_{\tr}\rangle}_{(c)}=0 \ .
     \end{equation}
     To do so, we will use a basis of $\Lsl(2,\R)$, namely we can write $\Lsl(2,\R)=\Span_{\R}(\xi_1,\xi_2,\xi_3\}$ where $$\xi_1=J_0,\qquad\xi_2=\begin{pmatrix}
     1 & 0 \\ 0 &-1
     \end{pmatrix},\qquad \xi_3=\begin{pmatrix}
     0 & 1 \\ 1 & 0
     \end{pmatrix} \ . $$
     The only symmetric matrices of this basis are $\xi_2$ and $\xi_3$, hence it is sufficient to prove equation (\ref{lastequation}) when $X^s=\xi_2$ and $X^s=\xi_3$, since all the elements are linear in $X\in\Lsl(2,\R)$. In both cases we use the description in coordinates $z=x+iy$ for $\Hyp$ and $w=u+iv$ for $\C$, of the Pick form $A$ and its variation $\dot A$, as we did at the end of Section 3.1. In particular we can do the computation in $(z,w)=(i,w)$ by $\SL(2,\R)$-invariance.\begin{itemize}
         \item[(i)] $X^s=\xi_2$.\newline In this case if $\{\frac\partial{\partial x_0},\frac\partial{\partial y_0}\}$ is a $g_{J_0}$-orthonormal basis of $\R^2$, then $X^s\cdot\frac\partial{\partial x_0}=\frac\partial{\partial x_0}$ and $X^s\cdot\frac\partial{\partial y_0}=-\frac\partial{\partial y_0}$, hence $A(X^s\cdot)=A_1\dx_{0}-A_2\dy_{0}$. 
         In particular, 
         \begin{align*}
             \tr(A_1(\dot A_1)_0J_0)&=2(-|w|^2\dot x-v\dot u+u\dot v) \\
             \tr(A_2(\dot A_2)_0J_0)&=2(-2|w|^2\dot x+u\dot v-v\dot u) \ . 
        \end{align*}
        Hence, \begin{align*}\frac{f'}{3}\langle A(X^s\cdot),(\dot AJ+A\dot J)_0\rangle&=\frac{f'}{3}\int_{T^2}\tr(A_1(\dot A_1)_0J-A_2(\dot A_2)_0J)\dx_0\wedge\dy_0 \\&=\frac 2{3}f'|w|^2\dot x \ .\end{align*}
     On the other hand, since \begin{align*}
         &\tr(\dot A_1)=-2(u\dot y+v\dot x)\qquad \tr(A_1X^s)=2u \\ &\tr(\dot A_2)=2(u\dot x-v\dot y) \qquad \tr(A_2X^s)=2v ,
     \end{align*}we get \begin{align*}
     -\frac{f'}{6}\langle AX^s+&X^sA, (\dot AJ+A\dot J)_{\tr}\rangle \\
     &= -\frac{f'}{6}\int_{T^2}\bigg(\tr(A_1X^s)\tr(\dot A_2)-\tr(A_2X^s)\tr(\dot A_1)\bigg)\dx_0\wedge\dy_0 \\&=-\frac 2{3}f'|w|^2\dot x \end{align*}and equation (\ref{lastequation}) is clearly satisfied.
     \item[(ii)] $X^s=\xi_3.$\newline With the same notation as in case $(i)$ we have $X^s\cdot\frac\partial{\partial x_0}=\frac\partial{\partial y_0}$ and $X^s\cdot\frac\partial{\partial y_0}=\frac\partial{\partial x_0}$, hence $A(X^s\cdot)=A_2\dx_{0}+A_1\dy_{0}$. Using that \begin{align*}
         &(\dot A_1)_0+(\dot A_2)_0J=\begin{pmatrix}
         -u\dot y-v\dot x & u\dot x-v\dot y \\ u\dot x-v\dot y & u\dot y+v\dot x
         \end{pmatrix}  \\ &\tr(A_1\big((\dot A_1)_0+(\dot A_2)_0J\big))=-2|w|^2\dot y
     \end{align*} the term $(a)$ multiplied by $\frac{f'}{3}$ is \begin{align*}
         \frac{f'}{3}\langle A(X^s\cdot), (\dot AJ+A\dot J)_0\rangle&=\frac{f'}{3}\int_{T^2}\tr(A_1\big((\dot A_1)_0+(\dot A_2)_0J\big))\dx_{0}\wedge\dy_{0} \\ &=-\frac 2{3}f'|w|^2\dot y \ .
     \end{align*}On the other hand, since \begin{align*}
         &\tr(\dot A_1)=-2(u\dot y+v\dot x)\qquad \tr(A_1X^s)=2v \\ &\tr(\dot A_2)=2(u\dot x-v\dot y) \qquad \tr(A_2X^s)=-2u ,
     \end{align*}the term $(c)$ multiplied by $-\frac{f'}{6}$  is \begin{align*}
        -\frac{f'}{6}\int_{T^2}\bigg(\tr(A_1X^s)\tr(\dot A_2)-\tr(A_2X^s)\tr(\dot A_1)\bigg)\dx_{0}\wedge\dy_{0}=\frac 2{3}f'|w|^2\dot y \ .
     \end{align*}
     \end{itemize}Thus, equation (\ref{lastequation}) is proved and the theorem as well. 
\end{proof}

\bibliographystyle{alpha}
\bibliographystyle{ieeetr}
\bibliography{biblio}

\end{document}